\documentclass{article}
\usepackage{amsthm,amsmath,amssymb}
\usepackage{graphics}
\usepackage{amsmath}
\usepackage{fancyhdr}
\usepackage{amsthm}
\usepackage{amsfonts}
\usepackage{eucal}
\usepackage{xspace}
\usepackage{makeidx}
\usepackage{enumerate}
\usepackage{subfigure}
\usepackage{multirow}
\usepackage{appendix}

\usepackage{amsfonts}
\usepackage{enumerate}
 \usepackage[utf8]{inputenc}
\usepackage{bbm}
\usepackage{mathrsfs}
\usepackage{color}
\usepackage{graphicx}
\usepackage{subfigure}
\usepackage{epsfig}

\usepackage[nottoc]{tocbibind}

\usepackage[top=1.45in, bottom=1.65in]{geometry}

\usepackage[colorlinks=true,citecolor=red,linkcolor=blue,%
urlcolor=RubineRed,pdfpagetransition=Blinds,pdftoolbar=false,pdfmenubar=false]{hyperref}

\theoremstyle{plain}
\newtheorem{theorem}{Theorem}

\newtheorem{lemma}[theorem]{Lemma}
\newtheorem{proposition}[theorem]{Proposition}
\newtheorem{definition}[theorem]{Definition}
\theoremstyle{remark}
\newtheorem{remark}[theorem]{Remark}

\usepackage{pdfsync}

\usepackage{amssymb,amsmath,amsfonts,eucal,xspace,enumerate,pst-node}
\usepackage[english]{babel}

\newcommand{\debil}{\rightharpoonup}

\newcommand{\field}[1]{\ensuremath{\mathbb{#1}}}

\newcommand{\R}{\field{R} \xspace}

\newcommand{\eps}{\varepsilon}

\begin{document}

\title{Robust Stackelberg controllability for a parabolic equation}
\author{V\'{\i}ctor  {\sc Hern\'andez-Santamar\'{\i}a }\thanks{Depto. de Control Autom\'atico, CINVESTAV . E-mail: {\tt
vhernandez@ctrl.cinvestav.mx}. Supported by CONACyT and project   IN102116 of DGAPA, UNAM.} \ and Luz
{\sc de Teresa}\thanks{Instituto de Matem\'aticas, Universidad Nacional Aut\'onoma de
M\'exico, Circuito Exterior, C.U., 04510 D.F., M\'exico. E-mail: {\tt
ldeteresa@im.unam.mx}. Supported by project   IN102116 of DGAPA, UNAM. \
(Mexico).} }


\maketitle

\begin{abstract}
The aim of this  paper is to perform a Stackelberg strategy to control parabolic equations. We have one control, \textit{the leader}, that is responsible for a null controllability property; additionally, we have a control \textit{the follower} that solves a robust control objective. That means, that we seek for a saddle point of a cost functional. In this way, the follower  control is not sensitive to a broad class of external disturbances. As far as we know, the idea of combining robustness with a Stackelberg strategy is new in literature
 \end{abstract}

\section{Introduction}

Let $\Omega\subset\mathbb{R}^N$, $N\geq 1$ be a bounded open set with boundary $\partial \Omega\in C^2$. For $T>0$, we denote $Q=\Omega\times(0,T)$ and $\Sigma=\partial \Omega\times(0,T)$. Let $\omega$ and $\mathcal O$ be nonempty subsets of $\Omega$ with $\omega\cap\mathcal O=\emptyset$. We consider the semilinear heat equation
\begin{equation}\label{semi_heat}
\begin{cases}
y_t-\Delta y+f(y)=h\chi_{\omega}+v\chi_{\mathcal O}+\psi \quad\text{in }Q, \\
y=0\quad\text{on }\Sigma, \quad y(x,0)=y_0(x) \quad\text{in } \Omega.
\end{cases}
\end{equation}
where $f$ is a globally Lipschitz-continous function, $y_0\in L^2(\Omega)$ is a given initial data and $\psi\in L^2(Q)$ is an unknown perturbation. 

In \eqref{semi_heat}, $y=y(x,t)$ is the state and $h=h(x,t)$, $v=v(x,t)$ are two different control functions acting on the system through $\omega$ and $\mathcal O$, respectively.

We want to choose the controls $v$ and $h$ in order to achieve two different optimal objectives:
\begin{enumerate}
\item solve for the ``best'' control $v$ such that $y$ is ``not too far'' from a desired target $y_d$ which is effective even in the presence of the ``worst'' disturbance $\psi$, and
\item find the minimal $L^2$-norm control $h$ such that $y(\cdot, T)=0$.
\end{enumerate}

The first problem, introduced in \cite{temam} for the linearized Navier-Stokes system, looks for a control such that a cost functional achieves its minimum for the worst disturbance. Solving for such control is a way of achieving system robustness: a control which works even in the presence of the worst disturbance $\psi$ will also be robust to a class of other possible perturbations. This approach is useful in physical systems in which unpredictable disturbances are common. 

The second problem is a classical null controllability problem. It has been thoroughly studied in the recent years for a wide variety of systems described by partial differential equations, see for instance \cite{cara_guerrero}.

When dealing with multi-objective optimization problems, a concept of a solution needs to be clarified. There are different equilibrium concepts (see \cite{Nash,Pareto,Stackelber}) which determine a strategy leading to choice good controls. In the framework of control of PDEs, there are several works applying successfully these strategies, see, for instance, \cite{AMR,Glowinski,Guillen,Limaco,LionsHier,LionsSta}.

Here, we use the so-called hierarchic control introduced by Lions in \cite{LionsSta} to achieve the desired goals. This technique uses the notion of Stackelberg optimization. Below, we will explain both the robust control and the null controllability problem and then how we will apply the hierarchic control methodology to solve the multi-objective optimization problem. 

Using the idea of hierarchic control described above, we want to get $y(T)=0$ using the minimal $L^2$-norm control $h$ and to ``stay near'' a desired state $y_d$ with the control $v$ but there is now a perturbation affecting the performance of the system. 

In the case where $\psi=0$, this problem has been solved in \cite{araruna}.The aim of this paper is to   combine the concept of hierarchic control with the concept of robust control appearing in optimal control problems (see, for instance, \cite{temam,temam_nonlinear,aziz}). As far as we know, the idea of combining robustness with a Stackelberg strategy is new in literature. 

\subsection{The control problem}
Let $\mathcal O_d\subset \Omega$ be an open set representing an observation domain. Let us introduce the cost functional
\begin{equation}\label{func_robust}
J_{r}(\psi,v;h)=\frac{1}{2}\iint_{\mathcal O_d\times(0,T)}|y-y_d|^2dxdt+\frac{1}{2}\left[\ell^2\iint_{\mathcal O\times(0,T)}|v^2|dxdt-\gamma^2\iint_Q|\psi|^2dxdt\right].
\end{equation}
where $\ell,\gamma>0$ are constants and $y_d\in L^2(\mathcal O_d\times(0,T))$ is given. This functional describes the robust control problem. We seek to simultaneously maximize $J_r$ with respect to $\psi$ and minimize it with respect to $v$, while maintaining the state $y$ ``close enough'' to a desired target $y_d$ in $\mathcal{O}_d\times(0,T)$. Note that the functional \eqref{func_robust} generalizes some classical optimization problems (see, for instance, \cite{Lions_optim,trol}).

As explained in \cite{temam_nonlinear}, one can intuitively consider the problem as a game between a designer looking for the best control $v$ and a malevolent disturbance $\psi$ spoiling the control objective. The parameter $\ell^2$ may be interpreted as the price of the control to the designer: the $\ell\to \infty$ limit corresponds to a prohibitively expensive control and results in $v\to 0$ in the minimization with respect to $v$. On the other hand, the parameter $\gamma^2$ may be interpreted as the magnitude of the perturbation that the problem can afford. The $\gamma\to \infty$ limit results in $\psi\to 0$ in the maximization with respect to $\psi$. 

The robust control problem is considered to be solved when a saddle point $(\overline{v},\overline\psi)$ is reached. As we will see further, for $\gamma>\gamma_0$ and $\ell>\ell_0$, where $\gamma_0, \ell_0$ are some critical values, we obtain the existence and uniqueness of the saddle point. 

The second problem we aim to solve is to find the minimal norm control satisfying a null controllability constraint. More precisely, we look for a control $h\in L^2(\omega\times(0,T))$ minimizing 
\begin{equation}\label{func_lider}
J(h)=\frac{1}{2}\iint_{\omega_\times(0,T)}|h|^2dxdt\quad\text{subject to}\quad y(\cdot,T)=0.
\end{equation}
It is well-known that for nonlinear terms $f$ satisfying a global Lipschitz condition, system \eqref{semi_heat} is null controllable (see, for instance, \cite{cara_guerrero,fc_zuazua,fursi}). The proof combines an observability inequality for a suitable adjoint linear system and a fixed point technique. We will use a similar argument to deduce the null controllability within the hierarchic control framework. 

Now, we are in position to describe the hierarchic control strategy to solve the optimization problems associated to the cost functionals \eqref{func_robust} and \eqref{func_lider}. According to the formulation originally introduced by H. von Stackelberg \cite{Stackelber}, we denote $h$ as the \textit{leader} control and $v$ as the \textit{follower} control. 

First we assume that the state is well defined in function of the controls, the perturbation and the initial condition, that is, there exists $y=y(h,v,\psi)$ uniquely determined by $h$, $v$, $\psi$, and $y_0$. Then, the hierarchic control method follows two steps:
\begin{enumerate} 
\item The follower $v$ assumes that the leader $h$ has made a choice, that is, given $h\in L^2(\omega\times(0,T))$ we look for an optimal pair $(v,\psi)$ such that is a saddle point to \eqref{func_robust}. Formally defined:
\begin{definition}\label{def:robust}
Let $h\in L^2(\omega\times(0,T))$ be fixed. The control $\bar v\in \mathcal V_{ad}$, the disturbance $\bar\psi\in \Psi_{ad}$ and the associated state $\bar y=\bar y(h,\bar v,\bar \psi)$ solution to \eqref{semi_heat} are said to solve the robust control problem when a saddle point $(\bar \psi,\bar v)$ of the cost functional \eqref{func_robust} is achieved, that is 
\begin{equation}\label{sp_eq}
J_r(\bar v,\psi;h)\leq J_r(\bar v,\bar \psi;h)\leq J_r(v,\bar \psi;h), \quad \forall (v,\psi)\in \mathcal V_{ad}\times \Psi_{ad}.
\end{equation}
Here, $\mathcal V_{ad}$ and $\Psi_{ad}$ are non-empty, closed, convex, and bounded or unbounded sets defining the set of admissible controls and perturbations, respectively. 
\end{definition}
Under certain conditions, we will see that there exists a unique pair $(\bar v,\bar \psi)$ and $\bar y=\bar y(h,\bar v,\bar \psi)$ satisfying \eqref{sp_eq}. 

\item Once the saddle point has been identified for each leader control $h$, we look for an optimal control $\hat h$ such that
\begin{equation}\label{null_opt}
J_L(\hat h)=\min_{h}J_L(h)
\end{equation}
subject to
\begin{equation}\label{null_constraint}
\bar y(\cdot, T;h,\bar v(h), \bar \psi(h))=0.
\end{equation}
\end{enumerate}

\begin{remark}\label{rem:stack}
\begin{itemize}
\item As in \cite{LionsSta}, we use the hierarchic control strategy to reduce the original multi-objective optimization problem to solving the mono-objective problems \eqref{sp_eq} and \eqref{null_opt}--\eqref{null_constraint}. However, in the second minimization problem the optimal strategy of the follower is fixed and its characterization needs to be considered. Indeed, the follower anticipates the leader's strategy and reacts optimally to its action, then if the leader wants to optimize its objective it has to take into account the optimal response of the follower. 

\item Here, we use the fact that   $\omega\cap\mathcal O=\emptyset$. Observe that, in practice, the leader control  cannot decide explicitly what to
do at the points in the domain  of the follower. Indeed, if this assumption is not true,   
 once the  leader has been chosen, the follower is modifying the leader  at those points.
 \end{itemize}
\end{remark}

\subsection{Main results}

The first result concerning the robust hierarchic control is the following one:

\begin{theorem}\label{teo_main}\label{TEO_MAIN}  
Assume that $\omega\cap\mathcal O_d\neq \emptyset$ and $N\leq 6$. Let $f\in C^2(\mathbb R)$ be a globally Lipschitz function verifying $f(0)=0$ and $f^{\prime\prime}\in L^\infty(\mathbb R)$. Then, there exist $\gamma_0$, $\ell_0$ and a positive function $\rho=\rho(t)$ blowing up at $t=T$ such that for any $\gamma>\gamma_0$, any $\ell>\ell_0$, any $y_0\in L^2(\Omega)$, and any $y_d$ verifying 
\begin{equation}\label{yd_robust}
\iint_{\mathcal O_d\times(0,T)}\rho^2 |y_d|^2<+\infty, 
\end{equation}
there exist a leader control $h$ and a unique associated saddle point $(\bar v,\bar \psi)$ such that the corresponding solution to \eqref{semi_heat} satisfies \eqref{null_constraint}.
\end{theorem}

As usual in the robust control problems, the assumption on $\gamma$ means that the possible disturbances spoiling the control objectives must have moderate $L^2$-norms. Indeed, if this condition is not met we cannot prove the existence of the saddle point \eqref{saddle_point}. On the other hand, the assumption on the target $y_d$ means that it approaches $0$ as $t\to T$. This is a common feature in some null controllability problems (see, for instance, \cite{deteresa2000,araruna}). 

In the same spirit, we are interested in proving a hierarchic result when the follower control $v$ and the perturbation $\psi$ belong to some bounded sets. To this end, let $E_1$ and $E_2$ be two non-empty, closed intervals such that $0\in E_i$. We define the set of admissible controls by
\begin{equation}\label{V_ad}
\mathcal V_{ad}=\left\{v \in L^2(\mathcal O\times(0,T)): v(x,t)\in E_1 \ \text{for a.e.} \ (x,t)\in \mathcal O\times(0,T) \right\},
\end{equation}
and the set of admissible perturbations by 
\begin{equation}\label{Psi_ad}
 \Psi_{ad}=\left\{\psi \in L^2(Q): \psi(x,t)\in E_2 \ \text{for a.e.} \ (x,t)\in Q \right\}.
\end{equation}
Defined in this way, the sets $\mathcal V_{ad}$ and $\Psi_{ad}$ are non-empty, closed, convex, bounded sets of $L^2(\mathcal O\times(0,T))$ and $L^2(Q)$, respectively.

We will carry out the optimization problem in the set $\mathcal V_{ad}\times\Psi_{ad}$ and restrict ourselves to the linear case. The controllability result is the following:
\begin{theorem}\label{teo_robust_acotado}\label{TEO_ROBUST_ACOTADO}
Let us assume that $f(y)=ay$ for some $a=a(x,t)\in L^\infty (Q)$ and that $\omega\cap\mathcal O_d\neq \emptyset$. Then, there exist $\gamma_0$, $\ell_0$ and a positive function $\rho=\rho(t)$ blowing up at $t=T$ such that for any $\gamma>\gamma_0$, any $\ell>\ell_0$, any $y_0\in L^2(\Omega)$, and any $y_d\in L^2(\mathcal O_d\times(0,T))$ verifying \eqref{yd_robust}, there exist a leader control $h$ and a unique associated saddle point $(\bar v,\bar \psi)\in\mathcal V_{ad}\times\Psi_{ad}$ such that the corresponding solution to \eqref{semi_heat} satisfies \eqref{null_constraint}.
\end{theorem}
The above theorem allows us to consider more practical situations. In real-life applications, we may desire to constrain the controls due to the maximum and minimum limits of the actuators.  On the other hand, we would like to take into consideration the perturbations affecting the system from a family of functions a priori known, without the necessity to look for the optimal performance over a large set of disturbances. 
 
The rest of the chapter is organized as follows. In section \ref{robust_section}, we study the corresponding part to the robust control problem. In fact, we will see that provided a sufficiently large value of $\gamma$, there exists an optimal pair $(\bar v,\bar \psi)$ that can be chosen for any leader control. Then, in section 3, once the follower strategy has been fixed, we proceed to obtain the leader control $h$ verifying the null controllability problem. We devote section \ref{sec:acotados} to prove Theorem \ref{teo_robust_acotado}. 

\section{The robust control problem}\label{robust_section}
\subsection{Existence of the saddle point}
We devote this section to solve the minimization problem concerning the robust control problem. To do this, we will follow the spirit of \cite{temam_nonlinear}. Here, we present results needed to prove the existence and uniqueness of the saddle point, as well as its characterization. In this stage, we assume that the leader has made a choice $h$, so we will keep it fixed all along this section. 

It is well-known (see, for instance, \cite{lady}) that for a globally Lipschitz function $f$ and any $y_0\in L^2(\Omega)$, any $(h,v)\in L^2(\omega\times(0,T))\times L^2(\mathcal O\times(0,T))$ and any $\psi\in L^2(Q)$, system \eqref{semi_heat} admits a unique weak solution $y\in W(0,T)$, where 
\begin{equation}
W(0,T):=\left\{y\in L^2(0,T;H_0^1(\Omega)),\, y_t\in L^2(0,T;H^{-1}(\Omega))\right\}.
\end{equation}
Moreover, $y$ satisfies an estimate of the form 
\begin{equation}\label{est_cont_y}
\|y\|_{W(0,T)}\leq C\left(\|y_0\|_{L^2(\Omega)}+\|h\|_{L^2(\omega\times(0,T))}+\|v\|_{L^2(\mathcal O\times(0,T))}+\|\psi\|_{L^2(Q)}\right),
\end{equation}
where $C>0$ does not depend on $\psi$, $h$, $v$ nor $y_0$.
If, in addition, $y_0\in H_0^1(\Omega)$, then \eqref{semi_heat} admits a unique solution $y\in W^{2,1}_2(Q)$, where
\begin{equation}
W^{2,1}_2(Q):=\left\{y\in L^2(0,T;H^2(\Omega)\cap H_0^1(\Omega)), \, y_t\in L^2(0,T;L^2(\Omega))\right\}.
\end{equation}

\begin{remark}
In order to obtain a solution to \eqref{semi_heat} it is sufficient to consider a locally Lipschitz function $f$. Moreover, if $f$ satisfies a particular growth at infinity, system \eqref{semi_heat} is null controllable in the classical sense, see \cite{fc_zuazua}. At this level, the assumptions on $f$ in Theorem \ref{teo_main} are too strong. However, they will be essential for the differentiability results for the functional \eqref{func_robust}. 
\end{remark}
	
The main goal of this section is to proof the existence of a solution $(\bar v,\bar \psi)$ to the robust control problem of Definition \ref{def:robust}.  The result is based on the following:
\begin{proposition}\label{prop_saddle_ekeland}
 Let $J$ be a functional defined on $X\times Y$, where $X$ and $Y$ are convex, closed, non-empty, unbounded sets. If
\begin{enumerate}
\item $\forall v\in X$, $\psi\mapsto J(v,\psi)$ is concave and upper semicontinuous, 
\item $\forall \psi\in Y$, $v\mapsto J(v,\psi)$ is convex and lower semicontinuous, 
\item $\exists v_0\in X$ such that $\lim_{\|\psi\|_Y\to\infty}  J(v_0,\psi)=-\infty$,
\item $\exists \psi_0\in Y$ such that $\lim_{\|v\|_X\to\infty}  J(v,\psi_0)=+\infty$,
\end{enumerate}
then $\mathcal J$ possesses at least one saddle point $(\bar v, \bar \psi)$ and
\begin{equation*}
\mathcal J(\bar v,\bar \psi)=\min_{v\in X}\sup_{\psi\in Y} J(v,\psi)=\max_{\psi\in Y}\inf_{v\in X} J(v,\psi).
\end{equation*}
\end{proposition}

The proof can be found on \cite[Prop. 2.2, p. 173]{Ekeland}. We intend to apply Proposition \ref{prop_saddle_ekeland} to the functional \eqref{func_robust} with $X=L^2(\mathcal O\times(0,T))$ and $Y=L^2(Q)$. In order to establish conditions 1--4 for our problem, we need to study first the differentiability of the solution to \eqref{semi_heat} with respect to the data. We have the following results:

\begin{lemma}\label{lemma_frechet_1}
Let $f$ be as in Theorem \ref{teo_main} and $h\in L^2(\omega\times(0,T))$ be given. Then, the operator $G:(\psi,v)\to y$ solution to \eqref{semi_heat} is continuously Fr\'echet differentiable from $L^2(\mathcal{O}\times(0,T))\times L^2(Q)\mapsto W^{2,1}_2(Q)$. The directional derivate in every direction $(v^\prime,\psi^\prime)$ is given by 
\begin{equation}\label{map_frechet_1}
G'(v,\psi)(v^\prime,\psi^\prime)=w
\end{equation}
where $w$ is the solution to the linear system 
\begin{equation}\label{der_frechet_1}
\begin{cases}
w_t-\Delta w+f'(y)w=v^\prime\chi_{\mathcal{O}}+\psi^\prime, \quad \text{in } Q,\\
w=0\quad \text{on }\Sigma, \quad w(x,0)=0 \quad\text{in } \Omega.
\end{cases}
\end{equation}
with $y=G(v,\psi)$ solution to \eqref{semi_heat}.
\end{lemma}
\begin{proof}
We derive in a straightforward manner the first derivative of the operator $G$ and its characterization. The arguments used here are by now classic, see for instance \cite{zuazua_fer}, \cite{seidman}.

Given $(v^\prime,\psi^\prime)\in L^2(\mathcal O\times(0,T))\times L^2(Q)$ and $\tau\in (0,1)$, we will prove the Fr\'echet-differentiability of $G$ by showing the convergence of $w^\tau\to w$ as $\tau\to 0$ where $w^\tau:=(y^\tau-y)/\tau$ for $\tau\neq 0$, with $y=y(h,\psi,v)$, $y^\tau=y(h,v+\tau v^\prime,\psi+\tau \psi^\prime)$ and $w$ is the solution to the linear problem \eqref{der_frechet_1}. From a simple computation we obtain that $w^\tau$ satisfies 
\begin{equation}\label{Frechet_1}
\begin{cases}
w^\tau_t-\Delta w^\tau+\frac{1}{\tau}\left(f(y^\tau)-f(y)\right)=v^\prime\chi_{\mathcal{O}}+\psi^\prime, \quad \text{in } Q,\\
w^\tau=0\quad \text{on }\Sigma, \quad w^\tau(x,0)=0 \quad\text{in } \Omega.
\end{cases}
\end{equation}

Since $f$ is continuously differentiable, we can use the mean value theorem to deduce that 	
\begin{align}\label{g_tau}
g^{\tau}(x,t)&=\frac{1}{\tau}\left(f(y^\tau)-f(y)\right)=f'(\tilde y^\tau)w^\tau
\end{align}
where $\tilde y^\tau=y-\theta_\tau(y^\tau-y)$ with $\theta_\tau\in (0,1)$. Replacing \eqref{g_tau} in \eqref{Frechet_1} and then multiplying by $w^\tau$ in $L^2(\Omega)$, it is not difficult to see
\begin{align*}
\|w^\tau(t)\|_{L^2(\Omega)}+\|\nabla w^\tau\|_{L^2(0,T;H_0^1(\Omega))}\leq C\left(\|v^\prime\|_{L^2(0,T;L^2(\mathcal O))}+\|\psi^\prime\|_{L^2(Q)}\right), \\
\forall t\in [0,T], \quad \forall \tau\in (0,1).
\end{align*}
Hence, the sequence $\{w^\tau\}$ is bounded in $C([0,T];L^2(\Omega))\cap L^2(0,T;H_0^1(\Omega))$. Taking into account the above estimate, together with \eqref{Frechet_1}, \eqref{g_tau} and since $f$ is globally Lipschitz, we conclude that there exists a positive constant $C$ independent of $\tau$, $v^\prime$ and $\psi^\prime$ such that
\begin{equation*}
\|w^\tau\|_{W(0,T)}\leq C\left(\|v^\prime\|_{L^2(0,T;L^2(\mathcal O))}+\|\psi^\prime\|_{L^2(Q)}\right), \quad \forall \tau \in (0,1).
\end{equation*}
On the other hand, $w^\tau$ can be viewed as the solution of an initial boundary-value problem with right-hand side term $\tilde g:= v^\prime\chi_{\mathcal O}+\psi^\prime-g^\tau$ and zero initial datum. Thus, from classical energy estimates, 
\begin{equation*}
\|w^\tau\|_{W^{2,1}_2(Q)}\leq C\|\tilde g\|_{L^2(Q)}, \quad \forall \tau\in (0,1).
\end{equation*}
for some constant $C$ independent of $\tau$, $v^\prime$ and $\psi^\prime$. In view of the expression of $g^\tau$ and from the the global Lipschitz property of $f$, we deduce the existence of a positive constant still denoted by $C$, independent of $\tau$, $v^\prime$ and $\psi^\prime$, such that 
\begin{equation*}
\|w^\tau\|_{W^{2,1}_2(Q)}\leq C\left(\|v\|_{L^2(0,T;L^2(\mathcal O))}+\|\psi^\prime\|_{L^2(Q)}\right), \quad \forall \tau\in (0,1).
\end{equation*}

Since the space $W^{2,1}_2(Q)$ is reflexive, we have that (extracting a subsequence)
\begin{equation}\label{conv_wtau}
w^\tau\debil \hat w\quad \text{weakly in } W^{2,1}_2(Q),
\end{equation}
as $\tau\to 0$, for some element $\hat w$. Now, from the continuity of $y$ with respect to the data (see Eq. \eqref{est_cont_y}), and combining \eqref{conv_wtau} with the fact that $W^{2,1}_2(Q)\subset L^2(Q)$ with compact imbedding, we get
\begin{equation}\label{gtau}
g^\tau\to f^\prime(y)\hat w \quad\text{in } L^2(Q).
\end{equation}
Taking the weak limit in \eqref{Frechet_1} and using \eqref{gtau} is not difficult to see that $\hat w$ is solution to \eqref{der_frechet_1}. On the other hand, since $W^{2,1}_2(Q)\subset L^2(0,T;H_0^1(\Omega))$ with compact imbedding, we have that the convergence is strong in this space.  
\end{proof}

\begin{lemma}\label{lemma_frechet_2}
Under assumptions of Proposition \ref{lemma_frechet_1}. The operator $G:(\psi,v)\to y$ solution to \eqref{semi_heat} is twice continuously Fr\'echet differentiable from $L^2(\mathcal{O}\times(0,T))\times L^2(Q)\mapsto W(0,T)$. Moreover, the second derivative of $G$ at $(v,\psi)$ is given by the expression 
\begin{equation}\label{map_frechet_2}
G^{\prime\prime}(v,\psi)[(v_1,\psi_1),(v_2,\psi_2)]=z
\end{equation}
where $z$ is the unique weak solution to the problem 
\begin{equation}\label{frechet_segunda}
\begin{cases}
z_t-\Delta z+f^\prime(y)z=-f^{\prime\prime}(y)w_1w_2 \quad\textnormal{in } Q,\\
z=0\quad\textnormal{on } \Sigma, \quad z(x,0)=0 \quad\textnormal{in } \Omega.
\end{cases}
\end{equation}
with $y=G(v,\psi)$ solution to \eqref{semi_heat}, and where $w_i$ is the solution to \eqref{der_frechet_1} in the direction $(v_i,\psi_i)$.
\end{lemma}

\begin{proof}
We follow will the arguments of \cite{trol} and use the implicit function theorem to deduce that $G$ is twice continuously Fr\'echet differentiable. We rewrite system \eqref{semi_heat} as follows:
\begin{equation}\label{y_mapeos}
y=G_Q(v,\psi-f(y))+G_0(h,y_0)
\end{equation}
where $G_Q\in \mathcal{L}(L^2(\mathcal O\times(0,T))\times L^2(Q);W(0,T))$ and $G_0\in \mathcal L(L^2(\omega\times(0,T))\times L^2(\Omega);W(0,T))$. More specifically,  we rewrite $y$ as
\begin{equation}\label{y_split}
y=y_1+y_2
\end{equation}
where $y_1$ and $y_2$ are solution to
\begin{equation*}
\left.\begin{array}{cc}
\begin{cases}
y_{1,t}-\Delta y_1=v\chi_{\mathcal O}+\psi-f(y) \quad \text{in Q} \\
y_1=0\quad\text{on }\Sigma, \quad y_1(x,0)=0\quad\text{in } \Omega
\end{cases}, & 
\begin{cases}
y_{2,t}-\Delta y_2=h\chi_{\omega} \quad \text{in Q} \\
y_2=0\quad\text{on }\Sigma, \quad y_2(x,0)=y_0\quad\text{in } \Omega
\end{cases}
\end{array}\right.
\end{equation*}
Equivalently, we express equation \eqref{y_mapeos} in the form
\begin{equation*}
0=y-G_Q(v,\psi-f(y))+G_0y_0=:F(y,v,\psi)
\end{equation*}
In this way, $F$ is twice continuously Frechet differentiable from $L^2(\mathcal O\times(0,T))\times L^2(Q)\times L^2(\Omega)$ into $W(0,T)$. Indeed, $G_Q$ and $G_0$ are continuous linear mappings and the operator $y\mapsto f(y)$ is twice continuously Frechet differentiable. 

On the other hand, the derivative $\partial_y F(y,v,\psi)$ is surjective. In fact, 
\begin{equation*}
\partial_y F(y,v,\psi)=\overline w
\end{equation*}
is equivalent to 
\begin{equation*}
\overline{w}=\overline y+G_Q(0,-f'(y)\overline y).
\end{equation*}
Setting $\zeta=\overline y-\overline w$ and using the definition of the mapping $G_Q$, we have that the above equation is equivalent to the problem
\begin{equation}\label{partial_y}
\begin{cases}
\zeta_t-\Delta \zeta=f'(y)\zeta+f'(y)\overline w \quad\text{in } Q, \\
\zeta=0 \quad\text{on }\Sigma, \quad \zeta(x,0)=0 \quad\text{in } \Omega,
\end{cases}
\end{equation}
Thanks to the assumptions on $f$, for every $\overline w\in L^2(Q)$, system \eqref{partial_y} has a unique solution $\zeta\in W(0,T)$. Hence, by the implicit function theorem, the equation $F(y,v,\psi)=0$ has a unique solution $y=y(v,\psi)$ in some open neighborhood of any arbitrarily chosen point $(\widetilde y,\widetilde v,\widetilde\psi)$. Moreover, the implicit function theorem yields that $G$ inherits the smoothness properties of $F$, therefore $G$ is twice continuously Fr\'echet differentiable. 

To obtain the characterization of the second derivate, we note from \eqref{y_mapeos} that 
\begin{equation*}
y=G(v,\psi)=G_Q(v,\psi-f(G(v,\psi)))+G_0y_0.
\end{equation*}
Then, differentiating on both sides of the above equation with respect to $(v,\psi)$ in the direction $(v_1,\psi_1)$, we get
\begin{equation*}
G^\prime(v,\psi)(v_1,\psi_1)=-G_Q\big(f^\prime(G(v,\psi))[G^\prime(v,\psi)(v_1,\psi_1)]\big)+G_Q(v_1,\psi_1)
\end{equation*}
Repeating the process in the direction $(v_2,\psi_2)$ yields
\begin{align*}
G''(v,\psi)[(v_1,\psi_1),(v_2,\psi_2)]=&-G_Q\left\{f^{\prime\prime}(G(v,\psi))\left(G^\prime(v,\psi)(v_1,\psi_1)\right)\left(G^\prime(v,\psi)(v_2,\psi_2)\right)\right. \\
&\left.+f^\prime(G(v,\psi))G^{\prime\prime}(v,\psi)[(v_1,\psi_1),(v_2,\psi_2)]\right\}
\end{align*}
Setting $y=G(v,\psi)$, $w_i=G^\prime(v,\psi)(v_i,\psi_i)$ and $z=G^{\prime\prime}(v,\psi)[(v_1,\psi_1),(v_2,\psi_2)]$ in the previous equation, we obtain that
\begin{equation*}
z=-G_Q\left\{f^{\prime\prime}(y)w_1w_2+f'(y)z\right\}
\end{equation*}
Therefore, from the definition of $G_Q$, we conclude that $z$ is solution to \eqref{frechet_segunda}.
\end{proof}

With Lemmas \ref{lemma_frechet_1} and \ref{lemma_frechet_2}, we are ready to proof one of the main result of this section:

\begin{proposition}\label{propo_saddle}
Under assumptions of Lemma \ref{lemma_frechet_1}. Let $y_0\in L^2(\Omega)$ and $h\in L^2(\omega\times(0,T))$ be given. Then, for $\gamma$ and $\ell$ sufficiently large, there exists a saddle point $(\bar v,\bar \psi)\in L^2(\mathcal O\times(0,T))\times L^2(Q)$ and $\bar y=\bar y(h,\bar v,\bar \psi)$ such that  
\begin{equation}\label{saddle_point}
 J_r(\bar v,\psi;h)\leq  J_r(\bar v,\bar \psi;h)\leq  J_r(v,\bar \psi;h), \quad \forall (v,\psi)\in L^2(\mathcal O\times(0,T))\times L^2(Q).
\end{equation}
\end{proposition}
\begin{proof}

In order to prove the existence of the saddle point $(\bar v,\bar \psi)$ we will verify conditions 1--4 from Proposition \ref{prop_saddle_ekeland}. 

\textit{Condition 1}. By Lemma \ref{lemma_frechet_1}, and since the norm is lower semicontinuous, the map $\psi\mapsto \mathcal J(v,\psi)$ is upper semicontinuous. To check the concavity, we will show that 
\begin{equation*}
\mathcal G(\tau)=J_r\left(v,\psi+\tau\psi^\prime\right)
\end{equation*}
is concave with respect to $\tau$ near $\tau=0$, that is, $\mathcal G^{\prime\prime}(0)<0$. Using the notation previously introduced, we set $y:=G(v,\psi+\tau\psi^\prime)$. In view of the results of Lemmas \ref{lemma_frechet_1} and \ref{lemma_frechet_2}, we have that $\mathcal G(\tau)$ is a composition of twice differentiable maps. Then it can be readily verified that  
\begin{align*}
\mathcal G^\prime(\tau)=&\iint_{\mathcal O_d\times(0,T)}\left(G(v,\psi+\tau\psi^\prime)-y_d\right)G^\prime(v,\psi+\tau \psi^\prime)(0,\psi^\prime)dxdt \\
&-\gamma^2\iint_Q\left(\psi+\tau\psi^\prime\right)\psi^\prime dxdt.
\end{align*}
A further differentiation with respect to $\tau$ yields
\begin{equation}\label{g_2prima}
\begin{split}
\mathcal G^{\prime\prime}(\tau)=&\iint_{\mathcal O_d\times(0,T)}\left(G(v,\psi+\tau\psi^\prime)-y_d\right)G^{\prime\prime}(v,\psi+\tau\psi^\prime)\left[(0,\psi^\prime),(0,\psi^\prime)\right]dxdt \\
&+\iint_{\mathcal O_d\times(0,T)}\left|G^\prime(v,\psi+\tau\psi^\prime)(0,\psi^\prime)\right|^2dxdt-\gamma^2\iint_{Q}|\psi^\prime|^2dxdt.
\end{split}
\end{equation}
We define $y^\prime:=G(v,\psi+\tau\psi^\prime)(0,\psi^\prime)$ and $y^{\prime\prime}:=G^{\prime\prime}(v,\psi+\tau\psi^\prime)[(0,\psi^\prime),(0,\psi^\prime)]$ which, according to Lemmas \ref{lemma_frechet_1} and \ref{lemma_frechet_2}, are solution to 
\begin{align}\label{y_prima}
&\begin{cases}
y^\prime_t-\Delta y^\prime+f^\prime(y)y^\prime=\psi^\prime \quad\text{in }Q, \\
y^\prime=0 \quad\text{on } \Sigma, \quad y^\prime(x,0)=0 \quad\text{in } \Omega,
\end{cases} \\
&\begin{cases}\label{y_2prima}
y^{\prime\prime}_t-\Delta y^{\prime\prime}+f^{\prime}(y)y^{\prime\prime}=-f^{\prime\prime}(y)|y^\prime|^2 \quad\text{in }Q, \\
y^{\prime\prime}=0 \quad\text{on } \Sigma, \quad y^{\prime\prime}(x,0)=0 \quad\text{in } \Omega.
\end{cases}
\end{align}
Then, we rewrite \eqref{g_2prima} as
\begin{equation}\label{g_ys}
\mathcal G^{\prime\prime}(\tau)=\iint_{\mathcal O_d\times(0,T)}(y-y_d)y^{\prime\prime}dxdt+\iint_{\mathcal O_d\times(0,T)}|y^\prime|^2dxdt-\gamma^2\iint_Q|\psi^\prime|^2dxdt.
\end{equation}
Now, we will see that for sufficiently large $\gamma$ the last term in the above equation dominates and thus $\mathcal G^{\prime\prime}(0)<0$ for $(v,\psi)\in L^2(\mathcal O\times(0,T))\times L^2(Q)$.  

We begin by estimating the second term. Thanks to the assumptions on $f$, there exists $L>0$ be such that $|f^\prime(s)|+|f^{\prime\prime}(s)|\leq L$, $\forall s\in \mathbb{R}$. Since the linear system \eqref{y_prima} has a unique solution $y^\prime\in W(0,T)$ for any $\psi^\prime\in L^2(Q)$, we can obtain 
\begin{equation}\label{conc_E1}
\iint_{\mathcal O_d\times(0,T)}|y^\prime|^2dxdt\leq C_1\iint_Q |\psi^\prime|^2dxdt.
\end{equation}
for some $C_1>0$ only depending on $\Omega$, $\mathcal O_d$, $L$ and $T$.

To compute the first term, we need an estimate for $y^{\prime\prime}$. We multiply \eqref{y_2prima} by $y^{\prime\prime}$ in $L^2(\Omega)$ and integrate by parts, whence
\begin{align*}
\frac{1}{2}\frac{d}{dt}\int_{\Omega}|y^{\prime\prime}|^2dx+\int_\Omega |\nabla y^{\prime\prime}|^2dx&=-\int_\Omega f^\prime(y)|y^{\prime\prime}|^2dx-\int_\Omega f^{\prime\prime}(y)|y^\prime|^2y^{\prime\prime} \\
&\leq L\int_\Omega |y^{\prime\prime}|^2dx+L\int_\Omega |y^\prime|^2 |y^{\prime\prime}|,
\end{align*}
and using Gronwall's and Poincare's inequality, we obtain
\begin{align}\label{est_ybi}
\iint_Q|y^{\prime\prime}|^2dxdt\leq C\iint_Q|y^\prime|^2|y^{\prime\prime}|dxdt.
\end{align}
Applying H\"older inequality in the above expression yields 
\begin{equation}\label{est_ybi_holder}
\iint_Q|y^{\prime\prime}|^2dxdt\leq C \|y^\prime\|_{L^{2p^\prime}(0,T;L^{2q^\prime}(\Omega))}^2\|y^{\prime\prime}\|_{L^p(0,T;L^q(\Omega))}
\end{equation}
where $1/p+1/p^\prime=1$ and $1/q+1/q^\prime=1$. To bound the right-hand side of the previous inequality, the idea is to find $p$ and $q$ such that 
\begin{equation*}
y^{\prime\prime}\in L^p(0,T;L^q(\Omega)), \quad y^\prime\in L^{2p^\prime}(0,T;L^{2q^\prime}(\Omega))
\end{equation*}

First, recall that $y^\prime$ is more regular than $W(0,T)$. In fact, from classical results (see, for instance, \cite{evans}) we have that $y^\prime \in L^2(0,T; H^2(\Omega))\cap L^\infty(0,T;H_0^1(\Omega))$ with $y^\prime_t\in L^2(0,T;L^2(\Omega))$. Moreover, we have the estimate
\begin{equation}\label{est_W2}
\|y^\prime\|_{L^\infty(0,T;H_0^1(\Omega))}+\|y^\prime\|_{L^2(0,T;H^2(\Omega))}+\|y^\prime_t\|_{L^2(0,T;L^2(\Omega))}\leq C\|\psi^\prime\|_{L^2(Q)}.
\end{equation}
In view of \eqref{est_W2}, it is reasonable to look for conditions such that the following embedding holds
\begin{equation}\label{embed_deseado}
L^2(0,T;H^2(\Omega))\cap L^\infty(0,T;H_0^1(\Omega))\hookrightarrow L^{2p^\prime}(0,T;L^{2q^\prime}(\Omega)).
\end{equation}

Let $X$ and $Y$ be Banach spaces. From well-known interpolation results, we have 
\begin{equation}\label{embed_simon}
L^{p_0}(0,T;X))\cap L^{p_1}(0,T;Y)\hookrightarrow L^{p_\theta}(0,T;B), \quad \frac{1}{p_\theta}=\frac{1-\theta}{p_0}+\frac{\theta}{p_1},
\end{equation}
with $0<\theta<1$ and where $B$ is the intermediate space of class $\theta$ (with respect to $X$ and $Y$), that is, $B$ is the space verifying 
\[
\|g\|_B\leq C \|g\|_X^{1-\theta}\|g\|_{Y}^\theta, \quad \forall g\in X\cap Y, \quad 0<\theta<1,
\]
for some $C$.

From \eqref{embed_deseado} and \eqref{embed_simon}, we deduce that 
\begin{equation}\label{int_T}
\frac{1}{2p^\prime}=\frac{\theta}{2}.
\end{equation}
On the other hand, from classical Sobolev embedding results, we have 
\begin{align} \label{esp_X}
H^2(\Omega) &\hookrightarrow L^{\frac{2N}{N-4}}(\Omega), \\ \label{esp_Y}
H^1_0(\Omega) &\hookrightarrow L^{\frac{2N}{N-2}}(\Omega).
\end{align}
for some maximal $N$ to be determined. Then, the space $L^{2q^\prime}(\Omega)$ is an intermediate space with respect to \eqref{esp_X} and \eqref{esp_Y} if
\begin{equation}\label{int_q_T}
\frac{1}{2q^\prime}=\frac{(N-4)\theta}{2N}+\frac{(N-2)(1-\theta)}{2N}, \quad 0<\theta<1.
\end{equation}
Setting $p^\prime$ to a fixed value such that $\theta\in (0,1)$ and replacing \eqref{int_T} into \eqref{int_q_T}, we obtain 
\begin{equation}\label{qprime_value}
q^\prime=\frac{p^\prime N}{p^\prime(N-2)-2},
\end{equation}
and from \eqref{int_T} and \eqref{qprime_value}, we deduce that 
\begin{equation}\label{p_q}
p=p^\prime/(p^\prime-1) \quad\text{and}\quad q=p^\prime N/(2p^\prime+2).
\end{equation}
Thus, from \eqref{embed_deseado}, \eqref{est_ybi_holder} and estimate \eqref{est_W2}, we get
\begin{equation*}
\iint_Q|y^{\prime\prime}|^2dxdt\leq C \|\psi^\prime\|_{L^2(Q)}^2\|y^{\prime\prime}\|_{L^p(0,T;L^q(\Omega))}.
\end{equation*}
It remains to verify that $L^{p}(0,T;L^q(\Omega))\hookrightarrow L^2(0,T;L^2(\Omega))$. From \eqref{p_q}, it is not difficult to see that this is true if $2\leq p^\prime$ and $N\leq 4(p^\prime+1)/p^\prime$. Setting $p^\prime=2$, we get $N\leq 6$ and thus the estimate
\begin{equation}\label{conc_E2}
\|y^{\prime\prime}\|_{L^2(0,T;L^2(\Omega))}\leq C_2 \|\psi^\prime\|_{L^2(Q)}^2.
\end{equation}
Putting together \eqref{g_ys}, \eqref{conc_E1} and \eqref{conc_E2} yields
\begin{equation*}
\mathcal G^{\prime\prime}(0)\leq (C_2+C_1-\gamma^2)\|\psi^\prime\|^2_{L^2(Q)}, \quad \forall \psi^\prime\in L^2(Q), \ \psi^\prime\neq 0.
\end{equation*}
Therefore, for $\gamma$ large enough we have $\mathcal G^{\prime\prime}(0)<0$ and therefore $\psi\mapsto J(v,\psi)$ is strictly concave. 

\textit{Condition 2}. By Lemma \ref{lemma_frechet_1}, and since the norm is lower semicontinuous, the map $v\mapsto J(v,\psi)$ is lower semicontinuous. In order to show convexity, it is sufficient to prove that 
\begin{equation*}
\mathcal G(\tau)=J(v+\tau v^\prime,\psi)
\end{equation*}
is convex with respect to $\tau$ near $\tau=0$, that is, $\mathcal G^{\prime\prime}(0)>0$. Arguing as above, we obtain 
\begin{equation}\label{g_convex}
\mathcal G^{\prime\prime}(\tau)=\iint_{\mathcal O_d\times(0,T)}(y-y_d)y^{\prime\prime}dxdt+\iint_{\mathcal O_d\times(0,T)}|y^\prime|^2dxdt+\ell^2\iint_{\omega\times(0,T)}|v^\prime|^2dxdt.
\end{equation}
where we have denoted $y^\prime=G(v+\tau v^\prime,\psi)(v^\prime,0)$ and $y^{\prime\prime}=G^{\prime\prime}(v+\tau v^\prime,\psi)[(v^\prime,0),(v^\prime,0)]$. Note that estimates for $y^\prime$ and $y^{\prime\prime}$ can be obtained in the same way as in the proof of Condition 1 by putting $v^\prime$ instead of $\psi^\prime$ in \eqref{y_prima}--\eqref{y_2prima}. Then, it is not difficult to see that
\begin{equation*}
\mathcal G^{\prime\prime}(0)\geq (l^2-C_2-C_1)\|v\|_{L^2(\mathcal O\times(0,T))}^2, \quad\forall v\in L^2(\mathcal O\times(0,T)), \ v\neq 0.
\end{equation*}
Thus, under the assumption that $\ell$ is large enough, $v\mapsto J(v,\psi)$ is strictly convex. 

\textit{Condition 3.} Taking $v=0$ and using formulas \eqref{y_mapeos}--\eqref{y_split} for $y=y(0,\psi)$ we obtain
\begin{align*}
J_r(0,\psi;h)&=\iint_{\mathcal O_d\times(0,T)}|y_1+y_2-y_d|^2dxdt-\frac{\gamma^2}{2}\iint_Q|\psi|^2dxdt \\
&\leq -\frac{\gamma^{2}}{2}\|\psi\|_{L^2(Q)}^2+C\|\psi\|_{L^2(Q)}^2+C_3,
\end{align*}
where $C_3$ is a positive constant only depending on $y_0$, $h$ and $y_d$. Hence, for a sufficiently large value of $\gamma$, condition 3 holds. 

\textit{Condition 4.} Taking $\psi=0$ in \eqref{func_robust} we get
\begin{equation*}
 J_r(v,0;h)\geq \frac{\ell^2}{2}\iint_{\mathcal O\times(0,T)}|v|^2dxdt,
\end{equation*}
and condition 4 follows immediately. This ends the proof. 
\end{proof}
\subsection{Characterization of the saddle point}\label{sec:charac_saddle}
The existence of a saddle point $(\bar v,\bar \psi)$ for the functional $J_r$ implies that 
\begin{equation}\label{nes_saddle}
\frac{\partial J_{r}}{\partial v}(\bar v,\bar \psi)=0 \quad \text{and}\quad \frac{\partial J_{r}}{\partial \psi}(\bar v,\bar \psi)=0,
\end{equation}
so our task is to find such expressions. Indeed, is not difficult to see that 
\begin{align}\label{grad_v}
\left(\frac{\partial J_r}{\partial v}(v,\psi),(v_1,0)\right)&=\iint_{\mathcal O_d\times(0,T)}(y-y_d)w_v dxdt+\ell^2\iint_{\mathcal O\times(0,T)}vv_1 dxdt \\ \label{grad_psi}
\left(\frac{\partial J_r}{\partial \psi}(v,\psi),(0,\psi_1)\right)&=\iint_{\mathcal O_d\times(0,T)}(y-y_d)w_\psi dxdt-\gamma^2\iint_{Q}\psi \psi_1 dxdt
\end{align}
where $w_v$ and $w_\psi$ are the directional derivatives of $y$ solution to \eqref{semi_heat} in the directions $(v_1,0)$ and $(0,\psi_1)$, respectively. To determine the solution of the robust control, we define the adjoint state
\begin{equation}\label{adjoint_robust}
\begin{cases}
-q_t-\Delta q+f^\prime(y)q=(y-y_d)\chi_{\mathcal O_d} \quad\text{in }Q, \\
q=0 \quad\text{on }\Sigma, \quad q(x,T)=0 \quad\text{in } \Omega.
\end{cases}
\end{equation}
We have the following result:
\begin{lemma}
Let $y=y(h,v,\psi)\in W(0,T)$ be the solution to \eqref{semi_heat}. Let $w$ be the solution to \eqref{der_frechet_1} with $(v_1,\psi_1)\in L^2(\mathcal O\times(0,T))\times L^2(Q)$ and $q$ be the solution to \eqref{adjoint_robust}. Then
\begin{equation}\label{opt_cond}
\iint_{\mathcal O_d\times(0,T)}(y-y_d)w\,dxdt=\iint_{\mathcal O\times(0,T)}q v_1dxdt+\iint_Qq\psi_1dxdt
\end{equation}
\end{lemma}
\begin{proof}
We multiply \eqref{adjoint_robust} by $w$ in $L^2(Q)$ and integrate by parts, more precisely
\begin{align*}
\iint_Q(y-y_d)\chi_{\mathcal O_d}qdxdt&=\iint_{Q}\left(-q_t-\Delta q+f^\prime(y)q\right)w\, dxdt \\
&=-\int_\Omega qw\,dx\Big |_0^T +\iint_Qq\left(w_t-\Delta w+f^\prime(y)w\right)dxdt
\end{align*}
Upon substituting the initial data for $q$ and $w$ and the right-hand side of \eqref{der_frechet_1} in the above equation, we obtain \eqref{opt_cond}.
\end{proof}
Replacing \eqref{opt_cond} in \eqref{grad_v}, with $\psi_1=0$ and taking an arbitrary $v_1\in L^2(\mathcal O\times(0,T))$  we get
\begin{equation*}
\left(\frac{\partial J_r}{\partial v}(v,\psi),(v_1,0)\right)=\iint_{\mathcal O\times(0,T)}qv_1dxdt+\ell^2\iint_{\mathcal O\times(0,T)}vv_1dxdt, \quad \forall v_1\in L^2(\mathcal O\times(0,T)).
\end{equation*}
In particular, we deduce 
\begin{equation}\label{exp_grad_v}
\frac{\partial J_r}{\partial v}(v,\psi)=(q+\ell^2v)|_{\omega}.
\end{equation}
Analogously, from \eqref{opt_cond} and \eqref{grad_psi} with $v_1=0$ and $\psi_1\in L^2(Q)$ as arbitrary we have
\begin{equation*}
\left(\frac{\partial J_r}{\partial \psi}(v,\psi),(0,\psi_1)\right)=\iint_{Q}q\psi_1dxdt-\gamma^2\iint_{Q}\psi\psi_1dxdt, \quad \forall\psi_1\in L^2(Q).
\end{equation*}
whence
\begin{equation}\label{exp_grad_psi}
\frac{\partial J_r}{\partial \psi}(v,\psi)=q-\gamma^2\psi.
\end{equation}
\begin{proposition}\label{Prop:follower_robust}
Let $h\in L^2(\omega\times(0,T))$ and $y_0\in L^2(\Omega)$ be given. Let $(\bar v,\bar \psi)$ be a solution to the robust control problem in Definition \ref{def:robust}. Then
\begin{equation}\label{charac_saddle}
\bar v=-\frac{1}{\ell^2}q\,|_{\mathcal \omega} \quad\text{and}\quad \bar\psi=\frac{1}{\gamma^2}q
\end{equation}
where $q$ is found from the solution $(y,q)$ to the coupled system
\begin{equation}\label{foll_robust}
\begin{cases}
y_t-\Delta y+f(y)=h\chi_\omega-\frac{1}{\ell^2}q\chi_{\mathcal O}+\frac{1}{\gamma^2}q &\quad\textnormal{in }Q, \\
-q_t-\Delta q+f^\prime(y)q=(y-y_d)\chi_{\mathcal O_d} &\quad\textnormal{in }Q, \\
y=q=0 &\quad\textnormal{on }\Sigma, \\
y(x,0)=y_0(x), \quad q(x,T)=0 &\quad\textnormal{in }\Omega.
\end{cases}
\end{equation}
which admits a unique solution for sufficiently large $\gamma$ and $\ell$. 
\end{proposition}
\begin{proof}
The existence of the solution to the robust control problem is ensured by Proposition \ref{propo_saddle} provided the parameters $\gamma$ and $\ell$ are large enough. A necessary condition for $(\bar v,\bar \psi)$ to be a saddle point of $ J_r$ is given in \eqref{nes_saddle}, therefore from \eqref{exp_grad_v} and \eqref{exp_grad_psi} we conclude that \eqref{charac_saddle}--\eqref{foll_robust} holds.  

To check uniqueness assume that $(\bar v,\bar \psi)$ and $(\tilde v,\tilde \psi)$ are two different saddle points in $L^2(\omega\times(0,T))\times L^2(Q)$. Then, from the strict convexity and strict concavity proved in Proposition \ref{propo_saddle}, we have
\begin{equation*}
\mathcal J(\tilde v,\tilde \psi)<\mathcal J(\bar v,\tilde \psi)<\mathcal J(\bar v,\bar \psi)
\end{equation*}
On the other hand, 
\begin{equation*}
\mathcal J(\bar v,\bar \psi)<\mathcal J(\tilde v,\bar \psi)<\mathcal J(\tilde v,\tilde \psi)
\end{equation*}
These lead to a contradiction, and therefore the saddle point $(\bar v,\bar \psi)$ is unique. 
\end{proof}
Summarizing, what we found in this section is that given a leader control $h$, there exists a unique solution to the robust control problem stated in Definition \ref{def:robust}. Moreover, it is characterized by the coupled system \eqref{foll_robust}. However, this characterization added a second equation coupled to the original system, so we need to take into account system \eqref{foll_robust} to obtain a solution to the leader's minimization problem (see Remark \ref{rem:stack}). 

\section{The null controllability problem: the observability inequality}\label{sec:null_robust}

Once the optimal strategy for the follower control has been chosen (see Section \ref{sec:charac_saddle}), the next step in the hierarchic methodology is to obtain an optimal control $\hat h$ such that
\begin{equation}\label{null_opt_leader}
J(\hat h)=\min_{h}J(h) \quad \text{subject to}\quad   y(\cdot, T)=0.
\end{equation}
where $y$ can be found from the solution $(y,q)$ to \eqref{foll_robust}. We start by proving an observability inequality for the adjoint system to the linearized version of \eqref{foll_robust}
\begin{equation}\label{adj_robust}
\begin{cases}
-\varphi_t-\Delta \varphi+a\varphi=\theta\chi_{\mathcal O_d} &\quad\textnormal{in }Q, \\
\theta_t-\Delta\theta+c\theta=-\frac{1}{\ell^2}\varphi\chi_{\mathcal O}+\frac{1}{\gamma^2}\varphi &\quad\textnormal{in }Q, \\
y=q=0 &\quad\textnormal{on }\Sigma, \\
\varphi(x,T)=\varphi^T(x), \quad \theta(x,0)=0 &\quad\textnormal{in }\Omega.
\end{cases}
\end{equation}
where $a,c\in L^\infty(Q)$ and $\varphi^T\in L^2(\Omega)$. Such inequality will be the main tool to conclude the proof of Theorem \ref{teo_main}. 

The main result of this section is the following one:
\begin{proposition}\label{propo:obs_ineq}
Assume that $\omega\cap\mathcal O_d\neq \emptyset$ and that $\gamma$ and $\ell$ are large enough. There exist a positive constant C only depending on $\Omega$, $\omega$, $\mathcal O$, $\mathcal O_d$, $\|a\|_\infty$, $\|c\|_\infty$, and $T$, and a weight function $\rho=\rho(t)$ blowing up at $t=T$ only depending on $\Omega$, $\omega$, $\mathcal O_d$, $\|a\|_\infty$, $\|c\|_\infty$ and $T$ such that, for any $\varphi^T\in L^2(\Omega)$, the solution $(\varphi,\theta)$ to \eqref{adj_robust} satisfies
\begin{equation}\label{ineq_robust}
\int_{\Omega}|\varphi(0)|^2dx+\iint_Q\rho^{-2}|\theta|^2dxdt\leq C\iint_{\omega\times(0,T)}|\varphi|^2dxdt.
\end{equation}
\end{proposition}

We postpone the proof of this result until the end of this section. The main tool to prove Proposition \ref{propo:obs_ineq} is a well-known Carleman inequality for linear parabolic systems. 

First, let us introduce several weight functions that will be useful in the reminder of this section. We introduce a special function whose existence is guaranteed by the following result \cite[Lemma 1.1]{fursi}.
\begin{lemma}\label{eta_fursi}
Let $\mathcal{B}\subset\subset\Omega$ be a nonempty open subset. Then, there exists $\eta^0\in C^2(\overline{\Omega})$ such that
\begin{equation*}
\begin{cases}
\eta^0(x)>0 \quad \text{all } x\in \Omega, \qquad \eta^0|_{\partial\Omega}=0, \\
|\nabla \eta^0|>0 \quad \text{for all } x\in \overline{\Omega\backslash\mathcal{B}}.
\end{cases}
\end{equation*}
\end{lemma}

For  $\lambda>0$ a parameter, we introduce the weight functions
\begin{equation}\label{weights}
\alpha(x,t)=\frac{e^{4\lambda\|\eta^0\|_\infty}-e^{\lambda\left(2\|\eta^0\|_\infty+\eta^0(x)\right)}}{t(T-t)}, \quad \xi(x,t)=\frac{e^{\lambda(2\|\eta^0\|_\infty+\eta^0(x))}}{t(T-t)}.
\end{equation}
For $m\in \R$ and a parameter $s>0$, we will use the following notation to abridge estimates:
\begin{equation}\label{ab_carleman}
\begin{gathered}
I_m(s,\lambda;z):=\iint_{Q}e^{-2s\alpha}(s\xi)^{m-2}\lambda^{m-1}|\nabla z|^2+\iint_{Q}e^{-2s\alpha}(s\xi)^{m}\lambda^{m+1}|z|^2, \\
I_{m,\mathcal{B}}(s,\lambda;z):= \iint_{\mathcal{B}\times(0,T)}e^{-2s\alpha}(s\xi)^{m}\lambda^{m+1}|z|^2.
\end{gathered}
\end{equation}

First, we state a Carleman estimate, due to \cite{ima_yama}, for solutions to the heat equation:
\begin{lemma}\label{car_basic}
Let $\mathcal{B}\subset\subset\Omega$ be a nonempty open subset. For any $m\in \mathbb{R}$, there exist positive constants $s_m$, $\lambda_m$, and $C_m$ such that, for any $s\geq s_m$, $\lambda\geq\lambda_m$, $F\in L^2(Q)$ and every $z^0\in L^2(\Omega)$, the solution $z$ to 
\begin{equation*}\label{sys_car_1}
\begin{cases}
z_t-\Delta z=F &\quad\text{in } Q, \\
z=0 &\quad\text{on } \Sigma, \\
z(x,0)=z^0(x) &\quad\text{in } \Omega,
\end{cases}
\end{equation*}  satisfies
\begin{equation}\label{chap1_ineq_car_basic}
I_m(s,\lambda;z)\leq C_m\left(I_{m,\mathcal{B}}(s,\lambda;z)+\iint_{Q}e^{-2s\alpha}(s\lambda\xi)^{m-3}|F|^2dxdt\right).
\end{equation}
Furthermore, $C_m$ only depends on $\omega$, $\mathcal{B}$ and $m$ and $s_m$ can be taken of the form $s_m=\sigma_m(T+T^2)$ where $\sigma_m$ only depends on $\omega$, $\mathcal{B}$ and $m$. 
\end{lemma}
\begin{remark}
Note that by changing $t$ for $T-t$, Lemma \ref{car_basic} remains valid for linear backward in time systems. Therefore, we can apply it interchangeably in what follows. 
\end{remark}

The observability inequality \eqref{ineq_robust} is consequence of a global Carleman inequality and some energy estimates. We present below a Carleman inequality for the solutions to system \eqref{adj_robust}:
\begin{proposition}\label{car_inicial}
Under assumptions of Proposition  \ref{propo:obs_ineq}. There exist positive constants constant $C$ and $\sigma_2$ such that the solution $(\varphi,\theta)$ to \eqref{adj_robust} satisfies
\begin{equation}\label{car_sinlocales_robust}
\begin{split}
I_3(s,\lambda;\varphi)+I_3(s,\lambda;\theta) \leq C\iint_{\omega\times(0,T)}e^{-2s\alpha}s^7\lambda^8\xi^7|\varphi|^2.
\end{split}
\end{equation}
for any $s\geq s_2=\sigma_2(T+T^2+T^2(\|a\|_\infty^{2/3}+\|c\|_\infty^{2/3}+\|a-c\|_\infty^{1/2}))$, any $\lambda \geq C$ and every $\varphi^T\in L^2(\Omega)$. 
\end{proposition}
\begin{proof}
Hereinafter $C$ will denote a generic positive constant that may change from line to line. We start by applying Carleman inequality \eqref{chap1_ineq_car_basic} to each equation in system \eqref{adj_robust} with $m=3$, $\mathcal{B}=\omega^\prime\subset\subset \omega_0:=\omega\cap\mathcal O_d$ and add them up, hence
\begin{equation*}\begin{split}
I_3&(s,\lambda;\varphi)+I_3(s,\lambda;\theta) \\
&\leq C\left(I_{3,\omega^\prime}(s,\lambda;\varphi)+I_{3,\omega^\prime}(s,\lambda;\theta)+\iint_Q e^{-2s\alpha}|\theta\chi_{\mathcal O_d}|^2dxdt\right. \\
&\qquad+\iint_Qe^{-2s\alpha}|-\tfrac{1}{\ell^2}\varphi\chi_{\mathcal O}+\tfrac{1}{\gamma^2}\varphi|^2 \left.+\iint_Qe^{-2s\alpha}\|a\|_\infty^2|\varphi|^2+\iint_Qe^{-2s\alpha}\|c\|_\infty^2|\theta|^2\right)
\end{split}
\end{equation*}
Taking the parameter $s$ large enough we can absorb some of the lower order terms in the right-hand side of the above expression. More precisely,  there exists a constant $\sigma_1>0$, such that
\begin{equation*}\begin{split}
I_3&(s,\lambda;\varphi)+I_3(s,\lambda;\theta) \\
&\leq C\left(I_{3,\omega^\prime}(s,\lambda;\varphi)+I_{3,\omega^\prime}(s,\lambda;\theta)+\iint_Q e^{-2s\alpha}|\theta\chi_{\mathcal O_d}|^2 +\iint_Qe^{-2s\alpha}|-\tfrac{1}{\ell^2}\varphi\chi_{\mathcal O}+\tfrac{1}{\gamma^2}\varphi|^2\right)
\end{split}
\end{equation*}
is valid for every 
\begin{equation}\label{s_1}
s\geq s_1=\sigma_1(T+T^2+T^2(\|a\|_\infty^{2/3}+\|c\|_\infty^{2/3})).
\end{equation}
Then, taking the parameter $\lambda$ large enough we get
\begin{equation}\label{car_conlocales_robust}
\begin{split}
I_3&(s,\lambda;\varphi)+I_3(s,\lambda;\theta) \leq C\left(I_{3,\omega^\prime}(s,\lambda;\varphi)+I_{3,\omega^\prime}(s,\lambda;\theta)\right).
\end{split}
\end{equation}
for every $s\geq s_1$ and $\lambda\geq C$.

The next step is to eliminate the local term on the right hand side corresponding to $\theta$. We will reason out as in \cite{deteresa2000} and \cite{luz_manuel}. We consider a function $\zeta\in C_0^\infty(\mathbb{R}^N)$ verifying: 
\begin{gather} \label{zeta_def}
0\leq \zeta\leq 1 \text{ in } \Omega, \quad \zeta\equiv 1 \quad\text{in } \omega^\prime, \quad \textrm{supp}\, \zeta\subset \omega_0, \\ \label{zeta_prop}
\frac{\Delta \zeta}{\zeta^{1/2}}\in L^\infty(\Omega), \quad \frac{\nabla \zeta}{\zeta^{1/2}}\in L^\infty(\Omega)^N
\end{gather}
Such function exists. It is sufficient to take $\zeta=\tilde \zeta^4$ with $\tilde \zeta\in C_0^\infty(\Omega)$ veryfing \eqref{zeta_def}. 

Let $s\geq s_1$ with $s_1$ given in \eqref{s_1}. We define $u=e^{-2s\alpha}s^3\lambda^4\xi^3$. Multiplying the equation satisfied by $\varphi$ in \eqref{adj_robust} by $u\zeta\theta$, integrating by parts over $Q$ and taking into account that $u(x, 0)$ vanishes in $\Omega$ we obtain
\begin{equation}\label{est_local_theta}
\begin{split}
\iint_{Q}u\zeta|\theta|^2\chi_{\mathcal O_d}=&\iint_{Q}(a-c)\varphi\theta u\zeta+\iint_Q \varphi\theta \zeta\partial_t u-\iint_Q \varphi\theta \Delta (u\zeta)  \\
&-2\iint_Q \nabla(u\zeta)\cdot \nabla\theta\,\varphi+\frac{1}{\gamma^2}\iint_Q|\varphi|^2u\zeta \\
:=&I_1+I_2+I_3+I_4+I_5.
\end{split}
\end{equation}
Let us estimate each $I_i$, $1\leq i \leq 4$, we keep the last term as it is. From H\"older and Young inequalities, we readily obtain
\begin{equation}\label{est_I1}
I_1=\iint_{Q}(a-c)\varphi\theta u\zeta \leq \delta_1 \iint_{Q}u\zeta|\theta|^2+\frac{1}{4\delta_1}\|a-c\|_\infty^2\iint_{Q}u\zeta|\varphi|^2.
\end{equation}
for any $\delta_1>0$. Observe that 
\begin{align*}
|\partial _t u|&\leq 3s^3\lambda^4\xi^2\xi_te^{-2s\alpha}+2s^3\lambda^4\xi^3e^{-2s\alpha}s\alpha_t, \\
&\leq CTs^3\lambda^4\xi^4e^{-2s\alpha}+CTs^4\lambda^4\xi^5e^{-2s\alpha}, \\
&\leq CTs^4\lambda^4\xi^5e^{-2s\alpha},
\end{align*}
where we have used that $\alpha_t\leq CT\xi^2$. Then, we can estimate
\begin{align} \nonumber
|I_2|&\leq \iint_Q |\varphi|  |\theta|  |\partial_t u| \zeta\leq CT\iint_Q s^4\lambda^4\xi^5e^{-2s\alpha}|\varphi| |\theta|\zeta \\ \nonumber
&\leq \delta _2 \iint_Q u\zeta |\theta|^2+\frac{CT^2}{\delta_2}\iint_Q s^5\lambda^4\xi^7 e^{-2s\alpha}|\varphi|^2\zeta \\ \label{est_I2}
&\leq \delta _2 \iint_Q u\zeta |\theta|^2+\frac{C}{\delta_2}\iint_Q s^7\lambda^4\xi^7 e^{-2s\alpha}|\varphi|^2\zeta
\end{align} 
for any $\delta_2>0$, where we have used in the last line that $s\geq \sigma_1T$. 

In order to estimate $I_3$, we compute first
\begin{equation}\label{delta_local}
\Delta \left(e^{-2s\alpha}s^3\lambda^4\xi^3\zeta\right)=\Delta\left(e^{-2s\alpha}s^3\lambda^4\xi^3\right)\zeta+\Delta \zeta e^{-2s\alpha}s^3\lambda^4\xi^3+2\nabla(e^{-2s\alpha}s^3\lambda^4\xi^3)\cdot\nabla \zeta
\end{equation}
and
\begin{align} \label{delta_u}
&|\Delta\left(e^{-2s\alpha}s^3\lambda^4\xi^3\right)|\leq Ce^{-2s\alpha}s^5\lambda^6\xi^5, \\ \label{grad_u}
&|\nabla\left(e^{-2s\alpha}s^3\lambda^4\xi^3\right)|\leq Ce^{-2s\alpha}s^4\lambda^5\xi^4,
\end{align}
where the above inequalities follow from the fact that 
\[
\partial_i\alpha=-\partial_i\xi=-C\lambda\partial_i \eta^0 \xi \leq C\lambda\xi.
\]
Then, from \eqref{delta_local}--\eqref{grad_u} and using \eqref{zeta_prop}, we obtain
\begin{align*}
|I_3|\leq& C\iint_Q |\varphi| |\theta|e^{-2s\alpha}s^5\lambda^6\xi^5\zeta+C\iint_Q|\varphi| |\theta| e^{-2s\alpha} s^3\lambda^4\xi^3\zeta^{1/2} \\
&+C\iint_Q |\varphi| |\theta|e^{-2s\alpha}s^4\lambda^5\xi^4\zeta^{1/2}.
\end{align*}
Using H\"older and Young inequalities and \eqref{zeta_def} yield
\begin{align*}
|I_3| \leq& \delta_3\iint_Q u\zeta|\theta|^2+\frac{C}{\delta_3}\iint_{\omega_0\times(0,T)}e^{-2s\alpha}s^7\lambda^8\xi^7|\varphi|^2 \\
&+ \frac{C}{\delta_3}\iint_{\omega_0\times(0,T)}e^{-2s\alpha}s^3\lambda^4\xi^3|\varphi|^2+ \frac{C}{\delta_3}\iint_{\omega_0\times(0,T)}e^{-2s\alpha}s^5\lambda^6\xi^5|\varphi|^2
\end{align*}
for some $\delta_3>0$. Note that $\xi^{-1}\leq CT^2/4$, then, for any $\nu,\mu\in\mathbb N$ with $\nu\geq \mu$ we have
\begin{equation}\label{sxi_mu}
(s\xi)^\mu=s^\mu\xi^{\nu}\xi^{\mu-\nu} \leq Cs^\mu\xi^\nu(T^2/4)^{-(\mu-\nu)}\leq Cs^\nu\xi^\nu,
\end{equation}
since $s\geq CT^2$. Hence,
\begin{equation}\label{est_I3}
|I_3| \leq \delta_3\iint_Q u\zeta|\theta|^2+\frac{C}{\delta_3}\iint_{\omega_0\times(0,T)}e^{-2s\alpha}s^7\lambda^8\xi^7|\varphi|^2.
\end{equation}
Using \eqref{zeta_prop}, \eqref{grad_u} and \eqref{sxi_mu}, we estimate $I_4$ as
\begin{align}\nonumber
|I_4|&\leq C\iint_Qe^{-2s\alpha}\left(s^3\lambda^4\xi^3|\nabla \theta| |\varphi|\zeta^{1/2}+s^4\lambda^5\xi^4|\nabla \theta| |\varphi| \zeta\right)  \\ \label{est_I4}
&\leq \varepsilon\iint_Qe^{-2s\alpha}s\lambda^2\xi|\nabla \theta|^2+\frac{C}{\varepsilon}\iint_{\omega_0\times(0,T)}e^{-2s\alpha}s^7\lambda^8\xi^7|\varphi|^2
\end{align}
for $\varepsilon>0$. 

Setting $\delta_i=1/6$, $1\leq i\leq 3$, and $\varepsilon=\frac{1}{4C}$ with $C$ the constant in \eqref{car_conlocales_robust}, and upon substituting estimates \eqref{est_I1}-\eqref{est_I2} and \eqref{est_I3}--\eqref{est_I4} in \eqref{est_local_theta}, we obtain
\begin{equation}\label{est_final_local}
\begin{split}
\iint_{Q}e^{-2s\alpha}s^3\lambda^4\xi^3|\theta|^2\chi_{\mathcal O_d} \leq& C\iint_{\omega_0\times(0,T)}e^{-2s\alpha}\left[\|a-c\|_\infty^2s^3\lambda^4\xi^3|\varphi|^2+s^7\lambda^8\xi^7|\varphi|^2\right] \\
&+ \frac{1}{2C}\iint_Qe^{-2s\alpha}s\lambda^2\xi|\nabla \theta|^2+\frac{1}{\gamma^2}\iint_Qe^{-2s\alpha}s^3\lambda^4\xi^3|\varphi|^2.
\end{split}
\end{equation}
Thus, in view of \eqref{car_conlocales_robust}--\eqref{zeta_def} and \eqref{est_final_local}, we obtain 
\begin{align*}
\iint_Qe^{-2s\alpha}&\left(s\lambda^2\xi|\nabla\varphi|^2+s^3\lambda^4\xi^3|\varphi|^2\right)+\iint_Qe^{-2s\alpha}\left(s\lambda^2\xi|\nabla\theta|^2+s^3\lambda^4\xi^3|\theta|^2\right) \\
\leq& C\|a-c\|_\infty^2\iint_{\omega_0\times(0,T)}e^{-2s\alpha}s^3\lambda^4\xi^3|\varphi|^2+C\iint_{\omega_0\times(0,T)}e^{-2s\alpha}s^7\lambda^8\xi^7|\varphi|^2 \\
&+\frac{C}{\gamma^2}\iint_Qe^{-2s\alpha}s^3\lambda^4\xi^3|\varphi|^2.
\end{align*}
Taking $s\geq CT^2\|a-c\|_\infty^{1/2}$, the above inequality now reads
\begin{equation}\label{car_con_gamma}
\begin{split}
\iint_Qe^{-2s\alpha}&\left(s\lambda^2\xi|\nabla\varphi|^2+s^3\lambda^4\xi^3|\varphi|^2\right)+\iint_Qe^{-2s\alpha}\left(s\lambda^2\xi|\nabla\theta|^2+s^3\lambda^4\xi^3|\theta|^2\right) \\
\leq&C\iint_{\omega_0\times(0,T)}e^{-2s\alpha}s^7\lambda^8\xi^7|\varphi|^2 +\frac{C}{\gamma^2}\iint_Qe^{-2s\alpha}s^3\lambda^4\xi^3|\varphi|^2.
\end{split}
\end{equation}
for every $s\geq s_2$ with 
\begin{equation}\label{s_2}
s_2=\sigma_2(T+T^2+T^2(\|a\|_\infty^{2/3}+\|c\|_\infty^{2/3}+\|a-c\|_\infty^{1/2})).
\end{equation}
for some $\sigma_2$ only depending on $\Omega$, $\omega$ and $\mathcal O_d$. 

Observe that the last term in \eqref{car_con_gamma} has the same power of $s$, $\lambda$ and $\xi$ as in the corresponding term on the left-hand side. Thus, provided $\gamma$ is large enough, we can absorb it into the right-hand side. Finally, since $\textnormal{supp }\omega_0\subset \omega$, we obtain the desired inequality \eqref{car_sinlocales_robust}. Therefore the proof is complete. 
\end{proof}
Now, we are going to improve inequality \eqref{car_sinlocales_robust} in the sense that the weight functions do not vanish at $t=0$. First, let us consider the function 
\begin{equation*}\label{l_t}
l(t)=\begin{cases}
T^2/4 &\quad\text{for}\quad 0\leq t\leq T/2, \\
t(T-t) &\quad\text{for}\quad T/2\leq t\leq T,
\end{cases}
\end{equation*}
and the functions
\begin{equation*}\label{weights_rec}
\begin{split}
&\beta(x,t)=\frac{e^{4\lambda\|\eta^0\|_\infty}-e^{\lambda(2\|\eta^0\|_\infty+\eta^0(x))}}{l(t)}, \quad \phi(x,t)=\frac{e^{\lambda(2\|\eta^0\|_\infty+\eta^0(x))}}{l(t)}, \\
&\beta^*(t)=\max_{x\in\overline \Omega}\beta(x,t), \quad \phi^*(t)=\min_{x\in\overline \Omega}\phi(x,t).
\end{split}
\end{equation*}
With these definitions, we have the following:
\begin{proposition}\label{car_modificado_robust}
Let $s$ and $\lambda$  as in Proposition \ref{car_inicial} and $\ell, \gamma$ be large enough.Then there exists a positive constant $C$ depending on $\Omega$, $\omega$, $\omega_d$, $s$, $\lambda$, $\|a\|_\infty$, $\|c\|_\infty$ and $T$ such that 
\begin{equation}\label{car_mod_2_rob}
\begin{split}
\|\varphi(0)\|^2_{L^2(\Omega)}+\iint_Q e^{-2s\beta^*}(\phi^*)^{3}|\varphi|^2 dxdt&+\iint_Q e^{-2s\beta^*}(\phi^*)^{3}|\theta|^2dxdt \\
&  \leq C\iint_{\omega\times(0,T)}e^{-2s\beta}\phi^{7}|\varphi|^2dxdt ,
\end{split}
\end{equation}
for any $\varphi^T\in L^2(\Omega)$, where $(\varphi,\theta)$ is the associated solution to \eqref{adj_robust}.
\end{proposition}
\begin{proof}
The proof is standard and relies on several well-known arguments \cite{}. First, by construction $\alpha=\beta$ and $\xi=\phi$ in $\Omega\times(T/2,T)$, hence
\begin{equation*}
\begin{split}
\int_{T/2}^T\!\int_{\Omega}e^{-2s\alpha}\xi^3|\varphi|^2+\int_{T/2}^T\!\int_{\Omega}e^{-2s\alpha}\xi^3|\theta|^2 \\
=\int_{T/2}^T\!\int_{\Omega}e^{-2s\beta}\phi^3|\varphi|^2+\int_{T/2}^T\!\int_{\Omega}e^{-2s\beta}\phi^3|\theta|^2
\end{split}
\end{equation*}
Therefore, from \eqref{car_sinlocales_robust} and the definition of $\beta$ and $\gamma$ we obtain
\begin{equation}\label{est_tmedios_robust}
\begin{split}
&\int_{T/2}^T\!\int_{\Omega}e^{-2s\beta}\phi^3|\varphi|^2+\int_{T/2}^T\!\int_{\Omega}e^{-2s\beta}\phi^3|\theta|^2  \leq C\iint_{\omega\times(0,T)}e^{-2s\beta}\phi^7|\varphi_1|^2
\end{split}
\end{equation}
On the other hand, for the domain $\Omega\times(0,T/2)$, we will use energy estimates for system \eqref{adj_robust}. In fact, let us introduce a function $\eta\in C^1([0,T])$ such that
\begin{equation*}
\eta=1\text{ in } [0,T/2], \quad \eta=0 \text{ in } [3T/4,T], \quad |\eta^\prime(t)|\leq C/T.
\end{equation*}
Using classical energy estimates for $\eta\varphi$ solution to the first equation of system \eqref{adj_robust} we obtain
\begin{equation*}
\begin{split}
&\|\varphi(0)\|^2_{L^2(\Omega)}+\|\varphi\|^2_{L^2(0,T/2;H_0^1(\Omega))} \leq C\left(\frac{1}{T^2}\|\varphi\|^2_{L^2(T/2,3T/4;L^2(\Omega))}+\|\eta\theta\|^2_{L^2(0,3T/4;L^2(\Omega))}\right)
\end{split}
\end{equation*}
From the definition of $\eta$, Poincar\'e inequality  and adding $\|\theta\|^2_{L^2(0,T/2;L^2(\Omega))}$ on both sides of the previous inequality we have
\begin{equation}\label{est_f}
\begin{split}
&\|\varphi(0)\|^2_{L^2(\Omega)}+\|\varphi\|^2_{L^2(0,T/2;L^2(\Omega))}+\|\theta\|^2_{L^2(0,T/2;L^2(\Omega))} \\
&\smallskip \leq C\left(\|\varphi\|^2_{L^2(T/2,3T/4;L^2(\Omega))}+\|\theta\|^2_{L^2(T/2,3T/4;L^2(\Omega))}+\|\theta\|^2_{L^2(0,T/2;L^2(\Omega))}\right)
\end{split}
\end{equation}
In order to eliminate the term $\|\theta\|^2_{L^2(0,T/2;L^2(\Omega))}$ in the right hand side, we use standard energy estimates for the second equation in \eqref{adj_robust}, thus
\begin{align} \nonumber
\iint_{\Omega\times(0,T/2)}|\theta|^2&\leq C\left(\frac{1}{\gamma^4}\iint_{Q}|\varphi|^2+\frac{1}{\ell^4}\iint_{\mathcal O\times(0,T)}|\varphi|^2\right) \\ \label{estimado_2_robust}
&\leq \frac{C}{\min\{\gamma^4,\ell^4\}}\iint_Q|\varphi|^2
\end{align}
Replacing \eqref{estimado_2_robust} in \eqref{est_f} and since $\gamma$ and $\ell$ are large enough we obtain
\begin{equation}\label{esti_1_robust}
\begin{split}
\|\varphi(0)\|^2_{L^2(\Omega)}+&\|\varphi\|^2_{L^2(0,T/2;L^2(\Omega))}+\|\theta\|^2_{L^2(0,T/2;L^2(\Omega))} \\
& \leq C\left(\|\varphi\|^2_{L^2(T/2,3T/4;L^2(\Omega))}+\|\theta\|^2_{L^2(T/2,3T/4;L^2(\Omega))}\right)
\end{split}
\end{equation}
Using \eqref{est_tmedios_robust} to estimate the terms in the right hand side of \eqref{esti_1_robust} and taking into account that the weight functions are bounded in $[0,3T/4]$ we have the estimate 
\begin{equation*}\label{est_0tmedios}
\begin{split}
\|\varphi(0)\|_{L^2(\Omega)}^2&+\int_{0}^{T/2}\!\!\!\!\int_{\Omega}e^{-2s\beta}\phi^3|\varphi|^2+\int_{0}^{T/2}\!\!\!\!\int_{\Omega}e^{-2s\beta}\phi^3|\theta|^2\\
&\leq C\left(\iint_{\omega\times(0,T)}e^{-2s\beta}\phi^7|\varphi_1|^2\right).
\end{split}
\end{equation*}
This estimate, together with \eqref{est_tmedios_robust}, and the definitions of $\phi^*$ and $\beta^*$ yield the desired inequality \eqref{car_mod_2_rob}.
\end{proof}

\begin{proof}[Proof of Proposition \ref{propo:obs_ineq}] 
The observability inequality \eqref{ineq_robust} follows immediately from Proposition \ref{car_modificado_robust}. Indeed, let us set $s=s_2$ as in \eqref{s_2} and define $\rho(t)=e^{s\beta^*}$. Thus $\rho(t)$ is a non-decreasing strictly positive function blowing up at $t=T$ that depends on $\Omega$, $\omega$, $\mathcal O_d$, $\|a\|_\infty$, $\|c\|_\infty$ and $T$, but can be chosen independently of $\mathcal O$, $\ell$ and $\gamma$. 

We obtain energy estimates with this new function for $\theta$ solution to the second equation of \eqref{adj_robust}. More precisely 
\begin{align*}
\iint_Q \rho^{-2}|\theta|^2dxdt &\leq C\left(\frac{1}{\gamma^4}\iint_{Q}\rho^{-2}|\varphi|^2dxdt+\frac{1}{\ell^4}\iint_{\mathcal O\times(0,T)}\rho^{-2}|\varphi|^2dxdt\right) \\
&\leq C\iint_Q \rho^{-2}|\varphi|^2dxt
\end{align*}
Since $e^{-2s\beta}\phi^7\leq C$ for all $(x,t)\in Q$ and noting that the right hand side of the previous inequality is comparable to the left hand side of inequality \eqref{car_mod_2_rob} up to a multiplicative constant, we obtain \eqref{ineq_robust}. This concludes the proof of Proposition \ref{propo:obs_ineq}.
\end{proof}

\section{Proof of Theorem \ref{teo_main}}\label{sec:leader_robust}
In this section, we will end the proof of Theorem \ref{teo_main}. We have already determined an optimal strategy for the follower control (see Proposition \ref{Prop:follower_robust}). It remains to obtain an strategy for the leader control $h$ such that $(y,q)$ solution to \eqref{foll_robust} verifies $y(T)=0$.

 The proof is inspired by well-known results on the controllability of nonlinear systems (see, for instance, \cite{zuazua_ondas,deteresa2000,zuazua_fer,b_gb_r_2,zuazua_fabre})  where controllability properties for linear problems and suitable fixed point arguments are the main ingredients.

\begin{proof}[Proof of Theorem \ref{teo_main}] We start by proving the existence of a leader control $h$ for a linearized version of \eqref{foll_robust}. In fact, for given $a,c\in L^\infty(Q)$, $y_0\in L^2(Q)$ and $y_d\in L^2(\mathcal O_d\times(0,T))$, we consider the linear system
\begin{equation}\label{foll_lin_robust}
\begin{cases}
y_t-\Delta y+ay=h\chi_\omega-\frac{1}{\ell^2}q\chi_{\mathcal O}+\frac{1}{\gamma^2}q &\quad\textnormal{in }Q, \\
-q_t-\Delta q+cq=(y-y_d)\chi_{\mathcal O_d} &\quad\textnormal{in }Q, \\
y=q=0 &\quad\textnormal{on }\Sigma, \\
y(x,0)=y_0(x), \quad q(x,T)=0 &\quad\textnormal{in }\Omega.
\end{cases}
\end{equation}
and the corresponding adjoint system \eqref{adj_robust}. Then, the following result holds
\begin{proposition}\label{prop:approx_control_robust}
Assume that $\omega\cap\mathcal O_d\neq \emptyset$. Let $C$ and $\rho$ as in Proposition \ref{propo:obs_ineq}. For any $\varepsilon>0$, any $y_0\in L^2(\Omega)$, and any $y_d\in L^2(\mathcal O_d\times(0,T))$ such that 
\begin{equation*}
\iint_{Q} \rho^{2}|y_d|^2 dxdt <+\infty
\end{equation*}
there exists a leader control $h_\varepsilon\in L^2(\omega\times(0,T))$ such that the associated solution $(y_\varepsilon,q_{\varepsilon})$ to \eqref{foll_lin_robust} satisfies 
\begin{equation}\label{approx_epsilon}
\|y_\varepsilon(T)\|_{L^2(\Omega)}\leq \varepsilon
\end{equation}
Moreover, the controls $\{h_\varepsilon\}_{\varepsilon>0}$ are uniformly bounded in $L^2(\omega\times(0,T))$, namely
\begin{equation}\label{est_uniforme_epsilon}
\|h_\varepsilon\|_{L^2(\omega\times(0,T))}\leq \sqrt{C}\left(\|y_0\|_{L^2(\Omega)}+\|\rho y_d\|_{L^2(Q)}\right), \quad \forall \varepsilon>0.
\end{equation}
\end{proposition}

\begin{proof}
The proof is by now standard. For the sake of completeness, we sketch some of the steps. For any fixed $\varepsilon>0$, consider 
\begin{equation}\label{func_approx}
\mathcal F_{\varepsilon}(\varphi^T)=\frac{1}{2}\iint_{\omega\times(0,T)}|\varphi|^2dxdt+\varepsilon\|\varphi^T\|_{L^2(\Omega)}+\int_{\Omega}y_0\varphi(0)dx-\iint_{\mathcal O_d\times(0,T)} \theta y_d\,dxdt
\end{equation}
where $(\varphi,\theta)$ is the solution to \eqref{adj_robust} with initial data $\varphi^T\in L^2(\Omega)$. It can be  verified that \eqref{func_approx} is continuous and strictly convex. From H\"older and Young inequalities and using the observability inequality \eqref{ineq_robust} is not difficult to see that 
\begin{equation*}
\mathcal F_{\varepsilon}(\varphi^T)\geq \frac{1}{4}\iint_{\omega\times(0,T)}|\varphi|^2dxdt+\varepsilon\|\varphi^T\|_{L^2(\Omega)}-C\left(\|y_0\|^2_{L^2(\Omega)}+\|\rho y_d\|_{L^2(Q)}^2\right),
\end{equation*}
hence \eqref{func_approx} is also coercive. Consequently, $\mathcal F_\varepsilon$ reaches its minimum at a unique point $\varphi^T_\varepsilon\in L^2(\Omega)$. When $\varphi_\varepsilon^T\neq 0$, the optimality condition can be computed, that is
\begin{equation}\label{opt_cond_aprox}
\begin{split}
&\iint_{\omega\times(0,T)}\varphi_\varepsilon\varphi\, dxdt+\left(\frac{\varphi_{\varepsilon}^T}{\|\varphi_\varepsilon^T\|},\varphi^T\right)_{L^2(\Omega)}\\
&+\int_{\Omega}y_0\varphi(0)dx-\iint_{\mathcal O_d\times(0,T)}y_{d}\theta\, dxdt=0, \quad\forall \varphi^T\in L^2(\Omega),
\end{split}
\end{equation}
where $(\varphi_\varepsilon,\theta_\varepsilon)$ is the solution to \eqref{adj_robust} with initial condition $\varphi_\varepsilon^T$. Set $h_\varepsilon=\varphi_\varepsilon \chi_{\omega}$, then $(y_\varepsilon, q_\varepsilon)$ solution to \eqref{foll_lin_robust} associated to this control verifies \eqref{approx_epsilon}. To conclude, observe that setting $\varphi^T=\varphi_\varepsilon^T$ in \eqref{opt_cond_aprox} and using the observability inequality \eqref{ineq_robust} yields estimate \eqref{est_uniforme_epsilon}.
\end{proof}

Now, we will apply a fixed point argument to prove an approximate controllability result for the nonlinear system \eqref{foll_robust}. For a given globally Lipschitz  function $f\in C^2(\mathbb R)$ verifying $f(0)=0$, we can write 
\begin{equation*}
f(s)=g(s)s, \quad \forall s\in \mathbb R,
\end{equation*}
where $g:\mathbb R\to\mathbb R$ is a continuous function defined by
\begin{equation*}
g(s)=\int_{0}^{1}f'(\sigma s)\, d\sigma.
\end{equation*}
The continuity of $f$ and $f'$ and the density of $C^\infty_c(Q)$ in $L^2(Q)$ allow  to see that $g(z)$ and $f'(z)$ belong to $L^\infty (Q)$ for every $z\in L^2(Q)$. 

For each $z\in L^2(Q)$, let us consider the linear system \eqref{foll_lin_robust} with $a=a_z=g(z)$ and $c=c_z=f'(z)$. Thanks to the hypothesis on $f$, there exists $M$ such that 
\begin{equation}\label{ac_z}
\|a_z\|_\infty, \|c_z\|_\infty \leq M, \quad \forall z\in L^2(Q)
\end{equation}
In view of Proposition \eqref{prop:approx_control_robust}, for any given $\varepsilon>0$ there exists a leader control $h_{z}\in L^2(\omega\times(0,T))$ such that the solution $(y_z,q_z)$ to \eqref{foll_lin_robust} corresponding to $a_z,c_z$ satisfies 
\[
\|y_z(T)\|_{L^2(\Omega)}<\varepsilon.
\]
Moreover, we have the estimate (uniform with respect to $\varepsilon$ and $z$)
\begin{equation}\label{est_uni_control}
\|h_z\|_{L^2(\omega\times(0,T))}\leq \sqrt C(\|y_0\|_{L^2(\Omega}+\|\rho y_d\|_{L^2(Q)}), \quad \forall z\in L^2(Q)
\end{equation}
where $C$ only depends on $\Omega$, $\mathcal O_d$, $\mathcal O$, $M$ and $T$ and $\rho$ only depends on $\Omega$, $\mathcal O_d$, $M$ and $T$.

We consider the mapping $\Lambda:L^2(Q) \to L^2(Q)$ defined by $\Lambda z=y_z$ with $(y_z,q_z)$ the solution to \eqref{foll_lin_robust} associated to the potentials $a_z$, $c_z$, and the control $h_z$ provided by Proposition \ref{prop:approx_control_robust}. By means of the Schauder fixed point theorem, we will deduce that $\Lambda$ possesses at least one fixed point. It can be proved that if $\ell$ and $\gamma$ are large enough then \eqref{foll_lin_robust} has a unique solution $y_z\in W(0,T)$ veryfing 
\begin{equation}\label{est_W}
\|y_z\|_{W(0,T)}\leq C\left(1+\|h\|_{L^2(\omega\times(0,T))}\right)
\end{equation}
where $C$ only depends on $\Omega$, $\mathcal O$, $\mathcal O_d$, $\gamma$, $\ell$, $K$, $y_0$, $y_d$ and $T$.  In view of \eqref{ac_z}--\eqref{est_W}, we deduce that $\Lambda$ maps $L^2(Q)$ into a bounded set of $W(0,T)$. This space is compacty embbeded in $L^2(Q)$, therefore it exists a fixed compact set $K$ such that
\begin{equation*}
\Lambda(L^2(Q))\subset K.
\end{equation*}

It can be readily verified that $\Lambda$ is also a continuous map from $L^2(Q)$ into $L^2(Q)$. Therefore, we can use Schauder fixed point theorem to ensure that $\Lambda$ has at least one fixed point $y=y_\varepsilon$, where $(y_\varepsilon,q_\varepsilon)$ together with the control $h_\varepsilon=h_{y_\varepsilon}$ solve
\begin{equation}\label{foll_lin_robust_epsilon}
\begin{cases}
y_{\varepsilon,t}-\Delta y_{\varepsilon}+g(y_\varepsilon)y_{\varepsilon}=h_\varepsilon\chi_\omega-\frac{1}{\ell^2}q_\varepsilon\chi_{\mathcal O}+\frac{1}{\gamma^2}q_\varepsilon &\quad\textnormal{in }Q, \\
-q_{\varepsilon,t}+\Delta q_\varepsilon+f'(y_\varepsilon)q_\varepsilon=(y_{\varepsilon}-y_d)\chi_{\mathcal O_d} &\quad\textnormal{in }Q, \\
y_\varepsilon=q_\varepsilon=0 &\quad\textnormal{on }\Sigma, \\
y_\varepsilon(x,0)=y_0 (x), \quad q_\varepsilon(x,T)=0 &\quad\textnormal{in }\Omega.
\end{cases}
\end{equation}
verifying \eqref{approx_epsilon}. 

To conclude the proof of Theorem \ref{teo_main}, we will pass to the limit in \eqref{foll_lin_robust_epsilon} and \eqref{approx_epsilon}. Thanks to \eqref{est_uni_control}, the control $h_\varepsilon$ is uniformly bounded in $L^2(\omega\times(0,T))$. Since \eqref{ac_z} holds, the solution $(y_\varepsilon,q_\varepsilon)$ lies in a bounded set of $W(0,T)\times W(0,T)$ and therefore in a compact set of $L^2(Q)\times L^2(Q)$. Then, up to a subsequence, we have
\begin{align*}
&h_\varepsilon\rightharpoonup h \quad\text{weakly in}\quad L^2(\omega\times(0,T)), \\
&(y_\varepsilon,q_{\varepsilon})\rightarrow (y,q)\quad \text{in}\quad L^2(Q)\times L^2(Q), \\
&y_\varepsilon(T)\to y(T) \quad \text{in}\quad L^2(\Omega),
\end{align*}
for some $h\in L^2(\omega\times(0,T))$ and some $(y,q)\in W(0,T)\times W(0,T)$. Due to the continuity of $g$, we can pass to the limit in \eqref{foll_lin_robust_epsilon}, thus $(y,q)$ solves \eqref{foll_robust} with leader control $h$ and initial datum $y_0$. Moreover, passing to the limit in \eqref{approx_epsilon} we conclude that $y(\cdot,T)=0$. Therefore the proof is complete. 
\end{proof}

\section{Proof of Theorem \ref{teo_robust_acotado}} \label{sec:acotados}
In the previous sections, we proved the existence of a robust Stackelberg control for a nonlinear system when $(v,\psi)\in L^2(\mathcal O\times(0,T))\times L^2(Q)$. Here, we will follow the arguments to show that a similar result can be obtained when the follower control $v$ and the perturbation $\psi$ belong to the bounded sets \eqref{V_ad}--\eqref{Psi_ad}, respectively. 

As stated in the theorem, we consider the linear system 
\begin{equation}\label{lin_sys_rob}
\begin{cases}
y_t-\Delta y+ay=h\chi_{\omega}+v\chi_{\mathcal O}+\psi \quad\text{in }Q, \\
y=0\quad\text{on }\Sigma, \quad y(x,0)=y_0(x) \quad\text{in } \Omega.
\end{cases}
\end{equation}
where $a\in L^\infty(Q)$ and $y_0\in L^2(\Omega)$ is given. 

It is clear that for given $y_0\in L^2(\Omega)$, any $h\in L^2(\omega\times(0,T))$ and each $(v,\psi)\in \mathcal V_{ad}\times \Psi_{ad}$, system \eqref{lin_sys_rob} admits a unique solution $y\in  C([0,T];L^2(\Omega))\cap L^2(0,T;H_0^1(\Omega))$. 

As before, we begin by proving the existence of a saddle point $(\bar v,\bar \psi)$ for the cost functional \eqref{func_robust}. The following result will give us conditions to determine its existence:

\begin{proposition}[Prop. 2.1, p. 171, \cite{Ekeland}]\label{prop_bounded}
 Let $ J$ be a functional defined on $X\times Y$, where $X$ and $Y$ are convex, closed, non-empty, bounded sets. If
\begin{enumerate}
\item $\forall v\in X$, $\psi\mapsto J(v,\psi)$ is concave and upper semicontinous, 
\item $\forall \psi\in Y$, $v\mapsto J(v,\psi)$ is convex and lower semicontinous, 
\end{enumerate}
then $J$ possesses at least one saddle point $(\bar v, \bar \psi)$ and
\begin{equation}
\mathcal J(\bar v,\bar \psi)=\min_{v\in X}\sup_{\psi\in Y} J(v,\psi)=\max_{\psi\in Y}\inf_{v\in X} J(v,\psi).
\end{equation}
\end{proposition}

We will apply Proposition \ref{prop_saddle_ekeland} to \eqref{func_robust} with $X=\mathcal V_{ad}$ and $Y=\Psi_{ad}$. In fact, verifying the conditions 1--2 will be easier than in the nonlinear case. Recall that in the first part of the hierarchic control the leader control $h$ is fixed. First, we have the following:
\begin{lemma}\label{lemma_lineal}
Let $h\in L^2(\omega\times(0,T))$ and $y_0\in L^2(\Omega)$ be given. The mapping $(v,\psi)\mapsto y(v,\psi)$ from $\mathcal V_{ad}\times \Psi_{ad}$ into $L^2(0,T;H_0^1(\Omega))$ is affine, continuous, and has G\^{a}teau derivative $y^\prime(v^\prime,\psi^\prime)$ in every direction $(v^\prime,\psi^\prime)\in L^2(\mathcal O\times(0,T))\times L^2(Q)$. Moreover, the derivative $y^\prime(v^\prime,\psi^\prime)$ solves the linear system 
\begin{equation}\label{lin_deriv}
\begin{cases}
y^\prime_t-\Delta y^\prime+ay^\prime=v^\prime\chi_{\mathcal O}+\psi^\prime \quad\text{in }Q, \\
y^\prime=0\quad\text{on }\Sigma, \quad y^\prime(x,0)=0 \quad\text{in } \Omega.
\end{cases}
\end{equation}
\end{lemma}
\begin{proof}
The fact that $(v,\psi)\mapsto y(v,\psi)$ is affine and continuous follows from the linearity of \eqref{lin_sys_rob} and well-known energy estimates for the heat equation. In the same way, thanks to linearity, the existence of the G\^ateau derivative and its characterization can be obtained by letting $\lambda$ tends to 0 in the expression $y^\lambda:=(y(v+\lambda v^\prime,\psi+\lambda\psi^\prime)-y(v,\psi))/\lambda$. 
\end{proof}

With this lemma, we are in position to check conditions 1--2 of Proposition \ref{prop_bounded}. This will give the existence of at most one saddle point of functional \eqref{func_robust}. 
\begin{proposition}\label{propo_saddle_lin}
Let $y_0\in L^2(\Omega)$ and $h\in L^2(\omega\times(0,T))$ be given. Then, for $\gamma$ sufficiently large, we have that  
\begin{enumerate}
\item $\forall\psi\in\Psi_{ad}$, $v\mapsto  J_r(v,\psi)$ is strictly convex lower semicontinuous, 
\item $\forall v\in \mathcal V_{ad}$, $\psi\mapsto  J_r(v,\psi)$ is strictly concave upper semicontinuous. 
\end{enumerate}
\end{proposition}
\begin{proof}
\textit{Condition 1}. Thanks to Lemma \ref{lemma_lineal}, the map $v\mapsto  J_r(v,\psi)$ is lower semicontinuous. Since  $v\mapsto y(v,\psi)$ is linear, the strict convexity of $ J_r$ can be readily verified. 

\textit{Condition 2}. Also, by Lemma \ref{lemma_lineal}, the map $\psi\mapsto  J_r(v,\psi)$ is upper semicontinuous. To prove the concavity, we will argue as in the nonlinear case. To this end, consider
\begin{equation*}
\mathcal G(\tau)= J_r(v,\psi+\tau\psi^\prime).
\end{equation*}
Then, it is sufficient to prove that $\mathcal G(\tau)$ is concave with respect to $\tau$. We compute
\begin{equation*}
\mathcal G^\prime(\tau)=\iint_{\mathcal O_d\times(0,T)} (y+\tau y^\prime-y_d)y^\prime-\gamma^{2}\iint_{Q}(\psi+\tau\psi^\prime)\psi^\prime
\end{equation*}
where $y^\prime$ is solution to \eqref{lin_deriv} with $v^\prime=0$. It is clear that $y^\prime$ is independent of $\tau$, hence
\begin{equation*}
\mathcal G^{\prime\prime}(\tau)=\iint_{\mathcal O_d\times(0,T)}|y^\prime|^2-\gamma^2\iint_Q|\psi^\prime|^2
\end{equation*}
From classical energy estimates for the heat equation, we obtain
\begin{equation*}
\mathcal G^{\prime\prime}(\tau)\leq -(\gamma^2-C)\|\psi^\prime\|_{L^2(Q)}^2, \quad \forall \psi^\prime\in L^2(Q)
\end{equation*}
where $C$ is a positive constant only depending $\Omega$, $\mathcal O_d$, $\|a\|_\infty$ and $T$. Then, for a sufficiently large value $\gamma$, we have $\mathcal G^{\prime\prime}(\tau)<0$, $\forall \tau\in \mathbb R$. Thus, the function $\mathcal G$ is strictly concave, and the strict concavity of $\psi\mapsto J_r(v,\psi)$ follows immediately. This concludes the proof. 
\end{proof}
Combining the statements of Propositions \ref{prop_bounded} and \ref{propo_saddle_lin}, we are able to deduce the existence of at most one saddle point $(\bar v,\bar \psi)\in \mathcal V_{ad}\times \Psi_{ad}$. Unlike the nonlinear case, the solution $(\bar v,\bar \psi)$ to the robust control problem may not necessarily satisfy \eqref{nes_saddle}, unless it is located in the interior of the domain $\mathcal V_{ad}\times \Psi_{ad}$. 

To characterize in this case the solution to the control problem, we use the fact that if $(\bar v,\bar \psi)$ is a saddle point of $J$, then
\begin{equation*}
J_r(\bar v,\bar \psi) \leq  J_r((1-\lambda)\bar v+\lambda v,\bar \psi), \quad \forall v\in \mathcal V_{ad}, 
\end{equation*}
or equivalently 
\begin{equation*}
0\leq  J_r(\bar v+\lambda(v-\bar v))- J_r(\bar v,\bar \psi), \quad \forall v\in \mathcal V_{ad}.
\end{equation*}
Dividing by $\lambda$ and taking the limit as $\lambda\to 0$, we obtain from the above expression
\begin{equation}\label{vin_control}
0\leq \iint_{\mathcal O_d\times(0,T)}(y-y_d)\hat y+{\ell^2}\iint_{\mathcal O\times(0,T)}\bar v(v-\bar v),
\end{equation}
where $y$ is the solution to \eqref{lin_sys_rob} evaluated in $(\bar v,\bar \psi)$ and $\hat y$ stands for the directional derivative \eqref{lin_deriv} in the direction $(v-\bar v,0)$. We introduce the adjoint state $q$ solution to the linear system 
\begin{equation}\label{adjoint_robust_lin}
\begin{cases}
-q_t-\Delta q+aq=(y-y_d)\chi_{\mathcal O_d} \quad\text{in }Q, \\
q=0 \quad\text{on }\Sigma, \quad q(x,T)=0 \quad\text{in } \Omega.
\end{cases}
\end{equation}
Multiplying \eqref{adjoint_robust_lin} by $\hat y$ and integrating by parts in $L^2(Q)$, it is not difficult to see that we can rewrite \eqref{vin_control} as 
\begin{equation*}
0\leq \iint_{\mathcal O\times(0,T)}(q+\ell^2 \bar v)(v-\bar v), \quad \forall v\in \mathcal V_{ad}.
\end{equation*}
Also, from the properties of the saddle point $(\bar v,\bar \psi)$, we have
\begin{equation*}
 J_r(\bar v,(1-\lambda)\bar \psi+\lambda \psi) \leq \mathcal J (\bar v,\bar \psi), \quad \forall \psi\in \Psi_{ad}.
\end{equation*}
Arguing as above, we deduce that
\begin{equation}\label{vin_pert}
\iint_{\mathcal O_d\times(0,T)}(y-y_d)\tilde y-{\gamma^2}\iint_{\mathcal O\times(0,T)}\bar \psi(\psi-\bar \psi)\leq 0,
\end{equation}
where $y$ is the solution to \eqref{lin_sys_rob} evaluated in $(\bar v,\bar \psi)$ and $\tilde y$ denotes the directional derivative \eqref{lin_deriv} in the direction $(0,\psi-\bar \psi)$. If we multiply \eqref{adjoint_robust_lin} by $\tilde y$ and integrate by parts in $L^2(Q)$, we can rewrite \eqref{vin_pert} as
\begin{equation*}
\iint_{Q}(q-\gamma^2\bar\psi)(\psi-\bar \psi)\leq 0, \quad \forall \psi\in \Psi_{ad}.
\end{equation*}
In this way, we have that $(\bar v,\bar \psi)$ satisfies the robust control problem \eqref{sp_eq} if $(y,p,\bar v,\bar \psi)$ satisfies the following optimality system:
\begin{align}\label{foll_robust_bound}
&\begin{cases}
y_t-\Delta y+ay=h\chi_\omega+\bar v\chi_{\mathcal O}+\bar \psi \quad\textnormal{in }Q, \\
-q_t-\Delta q+aq=(y-y_d)\chi_{\mathcal O_d} \quad\textnormal{in }Q, \\
y=q=0 \quad\textnormal{on }\Sigma, \quad y(x,0)=y_0(x), \ q(x,T)=0 \quad\textnormal{in }\Omega, \\
\end{cases} 
\\
&\quad \bar v\in \mathcal V_{ad}, \quad \bar \psi\in \Psi_{ad}, \\ \label{rob_opt_control}
&\ \iint_{\mathcal O\times(0,T)}(q+\ell^2 \bar v)(v-\bar v)\geq 0, \quad \forall v\in \mathcal V_{ad}, \\ \label{sys_opt_fin}
&\ \iint_{Q}(q-\gamma^2\bar\psi)(\psi-\bar \psi)\leq 0, \quad \forall \psi\in \Psi_{ad}.
\end{align}
From the hierarchic control methodology, the next step is obtain a leader control $h$ such that $y$ solution to the coupled system \eqref{foll_robust_bound} satisfies $y(T)=0$. The idea is to apply the results from Section \ref{sec:leader_robust}. We follow the spirit of \cite{araruna}.

First, note that $\bar \psi$ satisfying the variational inequality \eqref{sys_opt_fin} can be written as the projection onto the convex set $\Psi_{ad}$, that is,
\begin{equation*}
\bar \psi=\Pi_{\Psi_{ad}}\left(\tfrac{1}{\gamma^2}q\right).
\end{equation*}
The same is true for \eqref{rob_opt_control}. In this case, we have
\begin{equation*}
\bar v=\Pi_{\mathcal V_{ad}}\left(-\tfrac{1}{\ell^2}q|_{\mathcal O}\right).
\end{equation*}
In view of this, the optimality system \eqref{foll_robust_bound}--\eqref{sys_opt_fin} now reads
\begin{align}\label{foll_bound_rew}
&\begin{cases}
y_t-\Delta y+ay=h\chi_\omega+\Pi_{\mathcal V_{ad}}(-\tfrac{1}{\ell^2}q|_{\mathcal O})\chi_{\mathcal O}+\Pi_{\Psi_{ad}}(\tfrac{1}{\gamma^2}q) \quad\textnormal{in }Q, \\
-q_t-\Delta q+aq=(y-y_d)\chi_{\mathcal O_d} \quad\textnormal{in }Q, \\
y=q=0 \quad\textnormal{on }\Sigma, \\
y(x,0)=y_0(x), \ q(x,T)=0 \quad\textnormal{in }\Omega, \\
\end{cases} 
\end{align}
As in the semilinear case, we will analyze the null controllability of \eqref{foll_bound_rew} by means of a fixed point method. To do this, note that for every $z \in L^2(Q)$, $\Pi_{\Psi_{ad}}$ can be expressed in the form $\Pi_{\Psi_{ad}}(z)=\rho(z)z$ where the function $\rho(z)$ is defined as
\begin{equation*}
\rho(z)=
\begin{cases}
{1}, &\text{if } z(x,t)\in \Psi_{ad} \\
\Pi_{\Psi_{ad}}(z)/{z}, &\textnormal{otherwise}.
\end{cases}
\end{equation*}
for a.e. $(x,t)\in Q$.  

Defined in this way, $z\mapsto \rho(z)$ is continuous on $L^2(Q)$ and $\|\rho(z)\|_\infty\leq 1$, $\forall z\in L^2(Q)$. Analogously, we can define a function $\sigma$ such that $\Pi_{\mathcal V_{ad}}$ can be expressed in the form $\Pi_{\mathcal V_{ad}}=\sigma(z)z$ for every $z\in L^2(\mathcal O\times(0,T))$.

Therefore, the controllability problem is now to find $h\in L^2(\omega\times(0,T))$ such that the solution to 
\begin{align}\label{foll_bound_ls}
&\begin{cases}
y_t-\Delta y+ay=h\chi_\omega-\tilde\sigma(q)\tfrac{1}{\ell^2}q\chi_{\mathcal O}+\tilde\rho(q)\tfrac{1}{\gamma^2}q \quad\textnormal{in }Q, \\
-q_t-\Delta q+aq=(y-y_d)\chi_{\mathcal O_d} \quad\textnormal{in }Q, \\
y=q=0 \quad\textnormal{on }\Sigma, \\
y(x,0)=y_0(x), \ q(x,T)=0 \quad\textnormal{in }\Omega, \\
\end{cases} 
\end{align}
verifies $y(T)$=0. In system \eqref{foll_bound_ls}, $\tilde \sigma(q)$ stands for $\tilde \sigma(q)=\sigma(\tfrac{1}{\gamma^2}q|_\mathcal O)$ while $\tilde \rho(q)$ denotes $\tilde \rho(q)=\rho(\frac{1}{\gamma^2}q)$. We will establish the null controllability for \eqref{foll_bound_ls} arguing as in section \ref{sec:leader_robust}. 
 
For each  $\tilde q\in L^2(Q)$, let us consider the linear system
\begin{align}\label{y_qtlde}
&\begin{cases}
y_t-\Delta y+ay=h\chi_\omega-\tilde\sigma(\tilde q)\tfrac{1}{\ell^2}q\chi_{\mathcal O}+\tilde\rho(\tilde q)\tfrac{1}{\gamma^2}q \quad\textnormal{in }Q, \\
-q_t-\Delta q+aq=(y-y_d)\chi_{\mathcal O_d} \quad\textnormal{in }Q, \\
y=q=0 \quad\textnormal{on }\Sigma, \\
y(x,0)=y_0(x), \ q(x,T)=0 \quad\textnormal{in }\Omega. \\
\end{cases} 
\end{align}
In this case, adapting the arguments in Section \ref{sec:null_robust}, is not difficult to obtain an observability inequality (see Eq. \ref{ineq_robust}) for the solutions to the adjoint system 
\begin{equation*}\label{}
\begin{cases}
-\varphi_t-\Delta \varphi+a\varphi=\theta\chi_{\mathcal O_d} &\quad\textnormal{in }Q, \\
\theta_t-\Delta\theta+c\theta=-\frac{1}{\ell^2}\tilde \sigma(\tilde q)\varphi\chi_{\mathcal O}+\frac{1}{\gamma^2}\tilde \rho(\tilde q)\varphi &\quad\textnormal{in }Q, \\
y=q=0 &\quad\textnormal{on }\Sigma, \\
\varphi(x,T)=\varphi^T(x), \quad \theta(x,0)=0 &\quad\textnormal{in }\Omega.
\end{cases}
\end{equation*}

With this new observability estimate and following Section \ref{sec:leader_robust}, we can build a control $\tilde h$ associated to each $\tilde q\in L^2(Q)$ such that 
\begin{equation}\label{y_eps_tilde}
\|\tilde y(T)\|_{L^2(\Omega)}< \eps,
\end{equation}
where we have denoted by $\tilde y$ the first component of $(\tilde y,\tilde q)$ solution to \eqref{y_qtlde} with this control. Moreover, the control $\tilde h$ satisfies 
\begin{equation}\label{est_tilde_h}
\|\tilde h\|_{L^2(\omega\times(0,T))}\leq C,
\end{equation}
for some $C>0$ that can be chosen independently of $\gamma$ and $\ell$. 

Thanks to \eqref{est_tilde_h}, the controlled solution $(\tilde y,\tilde q)$ is uniformly bounded in $W(0,T)\times W(0,T)$. Therefore, we can deduce that the mapping $\tilde q\mapsto q$ has at least one fixed point. The rest of the proof follows as in Section \ref{sec:leader_robust}.
\bibliographystyle{amsalpha}

\end{document}